\documentclass[a4paper]{amsart} 
\usepackage[a4paper]{geometry}
\usepackage[utf8]{inputenc}
\usepackage{babel}
\usepackage{amsthm,amsmath,amssymb,amsfonts,amscd}
\usepackage{mathtools}
\usepackage{hyperref}
\usepackage{xcolor}
\usepackage{subfiles}
\usepackage{comment}
\usepackage{enumitem}
\usepackage{xparse}

\usepackage[textsize=tiny,textwidth=27mm]{todonotes}

\newtheorem{theorem}{Theorem}[section]
\newtheorem{lemma}[theorem]{Lemma}
\newtheorem{prop}[theorem]{Proposition}
\theoremstyle{definition}
\newtheorem{definition}{Definition}
\newtheorem*{remark}{Remark}

\newcommand{\R}{\ensuremath{\mathbb{R}}}

\newcommand{\T}{\ensuremath{\mathbb{T}}}

\newcommand{\Lipa}[1][\gamma]{\ensuremath{\mathcal{C}^#1}} 
\newcommand{\Na}[2][\gamma]{\ensuremath{\left\|#2\right\|_{#1}}} 
\newcommand{\SNa}[2][\gamma]{\ensuremath{\left|#2\right|_{#1}}} 
\newcommand{\Np}[2][\infty]{\ensuremath{\left\|#2\right\|_{L^#1}}}
 
\newcommand{\LipUa}[1][\gamma]{\ensuremath{\Lipa[{{1,#1}}]}}

\newcommand{\Lp}[2][2]{\ensuremath{\left\|#2\right\|_{L^#1(\T)}}}
\newcommand{\Ck}[2][2]{\ensuremath{\left\|#2\right\|_{C^#1(\T)}}} 
\newcommand{\Hs}[1]{\ensuremath{\left\|#1\right\|_{H^3(\T)}}} 
\newcommand{\Hd}[1]{\ensuremath{\left\|#1\right\|_{H^2(\T)}}} 

\NewDocumentCommand{\diff}{O{x'} O{x} m}{%
  \ensuremath{\Delta_{#1}#3 (#2)}%
}
\newcommand{\intT}{\ensuremath{\int_\T}} 
\newcommand{\intTT}{\ensuremath{\int_\T\int_\T}}

\author{Marc Magaña}
\address{Departament de Matemàtiques -- Universitat Autònoma de Barcelona\\
08193 Bellaterra, Barcelona, Spain}
\email{marc.magana@uab.cat}

\title{Unified theory for regularity persistence of vortex patch boundaries}

\begin{document}

\begin{abstract}
We establish a unified local theory for the persistence of Sobolev regularity of
vortex patch boundaries in a family of two-dimensional active scalar equations with
radial convolution kernels \(K(|x-y|)\). The class includes the 2D Euler equation,
the generalized SQG equation in the locally integrable range \(0<\beta<1\), and the
quasi-geostrophic shallow water equation. Under natural assumptions on \(K\)
(smoothness, integrability near the origin, monotonicity, and polynomial growth), we
prove that if the initial boundary belongs to \(H^3(\mathbb T)\) and satisfies the
arc-chord condition, then the contour dynamics equation admits a unique local solution
in \(C([0,T];H^3(\mathbb T))\). Under a stronger integrability condition on the kernel,
we also obtain local existence of \(H^2\) solutions. The proof combines Sobolev energy
estimates for the contour equation with quantitative control of the arc-chord quantity.
\end{abstract}

\maketitle

\section{Introduction}
We consider the Cauchy problem of the following 2D active scalar equation:
\begin{equation}\label{eq:transport}
    \left\{\begin{aligned}
        &\partial_t \omega + v\cdot \nabla \omega=0, \quad (t,x)\in \R^+\times \R^2\\ 
        &v= \nabla^\perp \psi, \\
        &\omega(0,\cdot)=\omega_0(x),\quad x\in \R^2,
    \end{aligned}\right.
\end{equation}
where $\nabla^\perp= (-\partial_2, \partial_1)^t,$ $v=(v_1,v_2)$ is the divergence free velocity field, $\omega$ is a scalar function (referred to as the active scalar) and the stream function $\psi$ is prescribed through a kernel $K$:
\begin{equation}\label{eq:stream}
    \psi(x)=\int_{\R^2} K(x,y)\omega(y)dy.
\end{equation}

The system \eqref{eq:transport}-\eqref{eq:stream} provides a unified framework that includes several classical hydrodynamic models, depending on the kernel $K$. Some important examples are given below.

\begin{itemize}
\item \textbf{2D Euler equation.}
If 
$$K(x,y) = -\frac{1}{2\pi}\log|x-y|, 
\text{ that is, } \psi(x) = (-\Delta)^{-1}\omega(x),$$
then \eqref{eq:transport}–\eqref{eq:stream} reduces to the 2D Euler equations in vorticity formulation. 
In this case, $\omega$ represents the fluid vorticity, and the system describes the evolution of an incompressible inviscid flow in two dimensions, one of the fundamental models in fluid dynamics.

\item \textbf{Generalized surface quasi-geostrophic (gSQG) equation.} For
$$K(x,y) = c_\beta |x-y|^{-\beta}, 
\; \beta \in (0,2),
\text{ that is, } \psi(x) = (-\Delta)^{-1+\frac{\beta}{2}}\omega(x),$$
where $c_\beta = \frac{\Gamma(\beta/2)}{2^{2-\beta}\pi\Gamma(1-\beta/2)},$ the system \eqref{eq:transport}–\eqref{eq:stream} corresponds to the inviscid gSQG. 
In particular, for $\beta=1$, one recovers the classical surface quasi-geostrophic (SQG) equation, introduced as a model for atmospheric circulation near the tropopause~\cite{sQG95} and for ocean dynamics in the upper layers~\cite{SQG2002}. 
For $0<\beta<1$, the model was proposed in~\cite{CordobaRodrigoaSQG} as an interpolation between the 2D Euler and the SQG equations.

\item \textbf{Quasi-geostrophic shallow water (QGSW) equation.} If 
$$K(x,y) = \frac{1}{2\pi}\mathbf{K}_0(\varepsilon |x-y|), 
\; \varepsilon>0,
\text{ that is, } \psi(x) = (-\Delta+\varepsilon^2)^{-1}\omega(x),$$
where $\mathbf{K}_0$ denotes the modified Bessel function of the second kind, then the system \eqref{eq:transport}–\eqref{eq:stream} becomes the QGSW equation. 
This model can be derived asymptotically from the rotating shallow water equations in regimes of rapid rotation and small surface variations~\cite{vallis2017atmospheric}. 
The parameter $\varepsilon$ is known as the \textit{inverse Rossby deformation length}, and the regime $\varepsilon \ll 1$ corresponds physically to a nearly rigid free surface.
\end{itemize}

Each of these models has been extensively studied, and many fundamental results are known regarding well-posedness, regularity, and long-time behavior. In particular, for the 2D Euler equation, Yudovich \cite{yudovich} proved that the vorticity equation is well-posed in \(L^\infty_c\). A \textit{vortex patch} is a special weak solution of this equation when the initial condition is the characteristic function of a bounded domain \(\Omega_0\). Since the vorticity equation is a transport equation, vorticity is conserved along trajectories, meaning that at any time \(t\), it remains the characteristic function of some domain \(\Omega(t)\).

More generally, for the family of active scalar equations \eqref{eq:transport}–\eqref{eq:stream}, we are interested in solutions of the same form, namely

\begin{equation}\label{eq:charDomini}
\omega(x,t)=\left\{\begin{array}{ll}
     1, & x\in\Omega(t), \\
     0, & \mathbb{R}^2\setminus\Omega(t),
\end{array}\right.
\end{equation}

where \(\Omega(t)\) is a bounded domain whose boundary is parametrized as a closed curve

\[
\partial\Omega(t)=\{\gamma(x,t)=(\gamma_1(x,t),\gamma_2(x,t))\,:\, x\in\T\},
\]

where \(\T=\mathbb{R}/(2\pi\mathbb{Z})\) denotes the one-dimensional torus, identified with \([-\pi,\pi]\) with periodic boundary conditions. As in the Euler case, the transport structure of \eqref{eq:transport} guarantees that if the initial datum is a characteristic function, the solution remains a characteristic function of some evolving domain for all times.

A fundamental question in the study of vortex patches is whether the initial regularity of the boundary persists over time. More precisely, if the boundary of the initial domain \(\partial\Omega_0\) belongs to a given regularity class (for instance, a Hölder space \(\mathcal{C}^{k,\sigma}\) or a Sobolev space \(W^{s,p}\)), one seeks to determine whether \(\partial\Omega(t)\) remains in the same class for all times, or at least locally in time.

For the 2D Euler equation in the plane, this problem was raised in the 1980s and resolved in the Hölder setting by Chemin \cite{chemin1993persistance}, who proved global persistence of \(\LipUa[\sigma]\) regularity (\(0<\sigma<1\)) using paradifferential calculus. Shortly thereafter, Bertozzi and Constantin \cite{bertozzi1993global} obtained an alternative proof based on geometric methods in classical analysis. Related propagation results for singular vortex patches were later obtained, for instance, in \cite{danchin1997singular}. These results were later extended to other domains, such as bounded domains and the half-plane \cite{depauw1999poche, kiselev2019global}. On the other hand, it has recently been shown by Kiselev and Luo that the Euler vortex patch problem is ill-posed in \(C^2\) \cite{kiselevLuo2022}, highlighting the delicate nature of the regularity threshold.

For the active scalar models introduced above, the persistence of boundary regularity has also been extensively studied. In the case of the SQG equation (\(\beta=1\)), Rodrigo \cite{rodrigo2005} proved local existence and uniqueness of \(C^\infty\) patches via a Nash--Moser implicit function theorem. For the generalized SQG equations (\(0<\beta<2\)), local-in-time persistence of Sobolev regularity for the patch boundary is by now well understood. Gancedo \cite{gancedo2008SQG} established local existence and uniqueness for patch boundaries parametrized in \(H^k(\T)\), \(k \geq 3\), for \(0<\beta<1\), and existence in the same class for \(\beta=1\). Uniqueness in the latter case was later proved by Córdoba, Córdoba, and Gancedo \cite{GancedoUniquenessSQG}. These results were subsequently extended by Chae, Constantin, Córdoba, Gancedo, and Wu \cite{chae2012generalized} to the full range \(0<\beta<2\), with \(k \geq 4\).

More recently, Gancedo and Patel \cite{gancedo2021local} obtained local existence for patch boundaries in \(H^2(\T)\) for \(0<\beta<1\) and in \(H^3(\T)\) for \(1<\beta<2\). In the critical case \(\beta=1\), Gancedo, Nguyen, and Patel \cite{GNP22} proved local well-posedness in \(H^{2+s}\) for arbitrarily small \(s>0\).

Several ill-posedness results are known for these models. Kiselev and Luo \cite{kiselevLuoAlphaSQG} proved strong ill-posedness of \(\mathcal{C}^{2,\sigma}\) and \(W^{2,p}\) (\(p\neq2\)) patch solutions for gSQG in the range \(0<\beta<1\). The formation of splash singularities, that is, configurations where the boundary self-intersects at a point while remaining otherwise smooth, has been ruled out for the full range of \(\beta\) in \cite{GS14, KLSplash23, MR4736524}. In the half-plane with a rigid boundary, finite-time singularity formation has been established for \(0<\beta<\tfrac{1}{3}\) (see \cite{MR3549626, MR3666567, gancedo2021local}; see also \cite{MR4993809} for further improvements).

For the quasi-geostrophic shallow water (QGSW) equation, Yudovich theory and global well-posedness of vortex patches in \(\LipUa[\sigma]\) was recently established in \cite{magana2026QGSW}. More generally, for the 2D loglog-Euler type equation, Tan, Xue, and Xue \cite{TXX2025} revisited Elgindi’s result \cite{MR3158812} and proved global propagation of patch boundary regularity with a loss of Hölder exponent, namely from \(\LipUa[\sigma]\) to \(\mathcal{C}^{1,\sigma-\varepsilon}\). They also pointed out that, under additional assumptions on the multiplier, this loss can be avoided. In the complementary regime, where the Osgood condition fails, Miao, Tan, Xue, and Xue \cite{MiaoTanXueXue2024} showed that finite-time singularities may occur for patch solutions in the half-plane.

Finally, persistence of boundary regularity has also been studied for a broader class of nonlinear transport equations in which the velocity field is given directly by convolution, \(v=K*\omega\), without imposing the divergence-free condition. This includes, for instance, the aggregation equation in \(\mathbb{R}^n\) with Newtonian kernel, \(K=\nabla N\), studied in \cite{bertozzi2012aggregation, bertozzi2016regularity}. Another important contribution is due to Cantero, Mateu, Orobitg, and Verdera \cite{cantero2021regularity}, who consider odd kernels in \(\mathbb{R}^n\), homogeneous of degree \(-(n-1)\) and smooth away from the origin, and prove global-in-time existence and uniqueness of patch solutions with \(\LipUa[\sigma]\) boundary, together with persistence of this regularity. Their uniqueness result is formulated within the class of characteristic functions of \(\LipUa[\sigma]\) domains, rather than as a general Yudovich-type theory for arbitrary bounded compactly supported data. Their framework allows for velocity fields with nontrivial divergence and includes, among other examples, kernels of the form \(L(\nabla N)\) where $L$ is a linear mapping from $\R^n$ into itself, such as the 3D quasi-geostrophic equation and the Cauchy transport equation in the plane. The planar case was later revisited by Verdera \cite{VerderaRevisited}, where the same global patch regularity result is obtained by different methods. For the Cauchy kernel, a particular case of the class in \cite{cantero2021regularity}, Clop and Sengupta \cite{clop2022nonlinear} proved existence of solutions for bounded, compactly supported initial data, even though the associated velocity field may have unbounded divergence, using compactness properties of quasiconformal mappings. In this setting, the main difficulty is that the velocity field may have non-zero, and in some cases unbounded, divergence.

In this paper, we propose a \textit{unified framework} encompassing a broad class of kernels \(K\) (including the examples above), and establish local-in-time persistence of Sobolev regularity for the patch boundary under explicit and verifiable assumptions on \(K\). Our approach is inspired by the work of Hmidi, Xue, and Xue \cite{hmidi2023unifiedtheoryvstatesstructures}, who developed a unified theory for V-state structures in active scalar equations.

More precisely, we consider radial kernels \(K:(0,\infty)\to\mathbb{R}\), namely kernels of the form
\[
K(x,y)=K(|x-y|),
\]
satisfying the following assumptions:

\begin{enumerate}[label={(A\arabic*)}]
    \item $K(r)\in C^\infty (\R\setminus \{0\}),$ \label{ass:smoothness}
    \item $\int_0^1 |K(r)| dr < \infty,$ \label{ass:integrability}
    \item There exists $ \delta >0$ such that $|K(r)|$ is decreasing on $(0,\delta]$. Moreover, for $r>\delta,$ there exists constants $C>0$ and $M\geq 0$ with $|K(r)|\leq C(1+r^M).$ \label{ass:compKernel}
    \item For $n=1,2,3$, $|K^{(n)}(r)|$ is decreasing for all $r>0$. \label{ass:monotonicity}
\end{enumerate}

The examples mentioned above (2D Euler, gSQG for $0<\beta<1$, QGSW) are all covered by our hypotheses. Further admissible kernels, including screened gSQG kernels and completely monotone radial kernels, are discussed in Appendix~\ref{app:examples}.

The evolution of the boundary \(\partial\Omega(t)\) is determined by the fact that the interface is transported by the flow. Let \(\gamma(\cdot,t)\) be a parametrization of \(\partial\Omega(t)\). Then the normal velocity of the boundary coincides with the normal component of the fluid velocity, namely
\[
(\partial_t \gamma - v(\gamma,t))\cdot n = 0,
\]
where \(n\) denotes the unit normal vector to the curve.

This identity determines only the normal component of \(\partial_t \gamma\), and therefore does not uniquely specify the evolution of the parametrization. Indeed, it is well known (see, for instance, \cite{gancedo2008SQG}) that the tangential component of the velocity does not affect the shape of the interface, but only its parametrization. Consequently, the evolution of the curve is determined up to the addition of an arbitrary tangential term.

Using a Biot-Savart type law, the velocity field can be expressed in terms of the boundary parametrization as
\[
v(\gamma,t) = \int_\mathbb{T} K(|\gamma(x,t)-\gamma(x-x',t)|)\,\partial_x \gamma(x-x',t)\,dx',
\]

Since only the normal component is relevant, we may exploit this freedom to rewrite the evolution equation in a more convenient form. Following \cite{gancedo2008SQG, gancedo2021local}, we consider the contour dynamics equation 
\begin{equation}\label{eq:CDEmod}\tag{CDE}
    \left\{
    \begin{aligned}
        &\partial_t \gamma(x,t) = -\int_\mathbb{T} K(|\gamma(x,t)-\gamma(x-x',t)|)\,(\partial_x \gamma(x,t)-\partial_x \gamma(x-x',t))\,dx'+ \lambda(x,t)\partial_x\gamma(x,t),\\
        &\gamma(x,0) = \gamma_0(x).
    \end{aligned}
    \right.
\end{equation}

where \(\lambda(x,t)\) is a scalar function chosen so that the parametrization satisfies a convenient condition (e.g., constant speed or a specific gauge).

To ensure that the curve remains well-behaved and does not self-intersect, we impose the arc-chord condition

\begin{equation}\label{eq:arcChord}
    \frac{|\gamma(x,t)-\gamma(x-x',t)|}{|x'|}\geq c >0,\quad \forall x,x'\in \mathbb{T},\ \forall t\geq 0,
\end{equation}

which rules out self-intersections and prevents the degeneration of the parametrization, in particular cusp-type collapses. Associated to this condition, one defines the quantity

\begin{equation}\label{eq:F}
    F(\gamma)(x,x',t)=\frac{|x'|}{|\gamma(x,t)-\gamma(x-x',t)|},
\end{equation}

with the convention \(F(\gamma)(x,0,t)=|\partial_x\gamma(x,t)|^{-1}\). We denote its supremum over the spatial variables by
\[
\Np{F(\gamma)}(t) \coloneqq \sup_{x,x'\in\mathbb{T}} F(\gamma)(x,x',t).
\]
The boundedness of \(F(\gamma)\), or equivalently the finiteness of \(\Np{F(\gamma)}\), will be a crucial control in the energy estimates.

Under these assumptions, we prove the following local well-posedness result.

\begin{theorem}[Local well-posedness in \(H^3\)]\label{thm:H3}
Let \(\gamma_0 \in H^3(\mathbb{T})\) be an initial curve such that \(F(\gamma_0)\) is bounded. Assume the radial kernel \(K\) satisfies hypotheses \ref{ass:smoothness}-\ref{ass:monotonicity}. Then there exists a time \(T > 0\) so that there is a unique solution to the contour dynamics equation \eqref{eq:CDEmod} in $C([0,T]; H^3(\T))$ with $\gamma(x,0) = \gamma_0(x)$ and $\lambda(x,t)=0$.
\end{theorem}

If the integrability condition \ref{ass:integrability} is strengthened, the required regularity can be lowered, although uniqueness at this level of regularity is currently unknown.

\begin{theorem}[Local existence in \(H^2\)]\label{thm:H2}
Let \(\gamma_0 \in H^2(\mathbb{T})\) be an initial curve such that \(F(\gamma_0)\) is bounded. Assume the radial kernel \(K\) satisfies \ref{ass:smoothness}, \ref{ass:compKernel} and \ref{ass:monotonicity} (only required for $n=1,2$) and the following stronger integrability condition:
\begin{equation}\label{ass:strong-integrability}\tag{A2'}
\int_0^1 r| K(r)|^{\frac{1}{\alpha}} \, dr < \infty \quad \text{for some } \alpha \in \left(0,\frac{1}{2}\right).
\end{equation}
Then there exists a time \(T > 0\) and a solution \(\gamma \in C([0,T]; H^2(\T))\) to the contour dynamics equation \eqref{eq:CDEmod} with $\lambda$ as defined in \eqref{eq:lambda}.
\end{theorem}

\begin{remark} \label{rmk:uniquenessH2}
For the specific family of gSQG kernels with $0<\beta<\frac{1}{3}$, Gancedo and Patel \cite{gancedo2021local} proved uniqueness of $H^2$
patches in the half‑plane setting. Their proof crucially exploits an extra cancellation coming from the reflected (mirror) terms that appear due to the fixed boundary condition. In the whole space, the same cancellation is absent, and uniqueness of $H^2$ patches for the gSQG equation (or for the general kernels covered by Theorem~\ref{thm:H2}) remains open to the best of our knowledge. 
\end{remark}

\begin{remark}
The assumptions on the kernel have been stated in a form that keeps the energy
estimates as transparent as possible. The smoothness assumption~\ref{ass:smoothness}
can be relaxed: it suffices that \(K\in C^3(\mathbb R\setminus\{0\})\) for
Theorem~\ref{thm:H3}, and \(K\in C^2(\mathbb R\setminus\{0\})\) for
Theorem~\ref{thm:H2}.

In a similar way, the global monotonicity assumption in~\ref{ass:monotonicity} is used mainly
to avoid repeated decompositions between small and large scales. One expects the same
arguments to work under a weaker condition in which, for the relevant derivatives,
\(|K^{(n)}|\) is required to be decreasing only near the origin, while away from the
origin the derivatives satisfy polynomial growth bounds analogous to
\ref{ass:compKernel}. Indeed, the singular part of the estimates only uses monotonicity
at small scales, whereas the contribution from the region
\(|\gamma(x)-\gamma(x-x')|\gtrsim 1\) can be controlled by Sobolev bounds on the curve
and polynomial growth of the kernel. We keep the stronger formulation in order to avoid
additional technical decompositions.
\end{remark}

The analytical framework developed in this paper follows closely the approach introduced by Gancedo in \cite{gancedo2008SQG} and further developed in \cite{gancedo2021local}. The argument is based on the contour-dynamics energy method introduced in the aforementioned works, but it is implemented here at the level of a general radial kernel \(K\). This requires isolating the precise kernel properties needed in the estimates and replacing the explicit computations available for the gSQG kernels by structural bounds valid for the whole class considered.

This work is organized as follows: Section~\ref{sec:preliminariesUnified} introduces the function spaces and notation used throughout the paper. Section~\ref{sec:H3} is devoted to the proof of Theorem~\ref{thm:H3}, and Section~\ref{sec:H2} to the proof of Theorem~\ref{thm:H2}. Finally, Appendix \ref{app:examples} discusses further examples of admissible kernels.

\section{Preliminaries}  \label{sec:preliminariesUnified}
In this section we introduce the function spaces and notation that will be used throughout the paper. We also recall the assumptions on the kernel and prove two key estimates that are fundamental for the energy arguments in the following sections.

\subsection{Function spaces and notation}
Let \(\mathbb{T}=[-\pi,\pi]\) be the one‑dimensional torus. Functions \(\gamma:\mathbb{T}\to\mathbb{R}^2\) are always assumed to be \(2\pi\)-periodic. For such functions we will use the standard Lebesgue and Sobolev spaces.

For an integer \(k\ge 0\), we define the Sobolev space \(H^k(\mathbb{T})\) as the set of functions whose derivatives up to order \(k\) belong to \(L^2(\mathbb{T})\). A standard norm is
$
\|\gamma\|_{H^k}^2 = \sum_{j=0}^k \|\partial_x^j\gamma\|_{L^2}^2.
$ However, for our energy estimates it will be convenient to use an equivalent norm that only involves the \(L^2\) norm of the function and its highest derivative. More precisely, on the torus we have the equivalence
\[
\|\gamma\|_{H^k(\T)}^2 \sim \Lp{\gamma}^2 + \Lp{\partial_x^k\gamma}^2,
\]
which follows from the Poincaré inequality: on the torus, any function can be decomposed into its constant part (controlled by \Lp{f}) and a zero-mean part whose lower-order derivatives are bounded by \Lp{\partial_x^k f}. This equivalence allows us to focus our estimates on the highest-order derivative, since the lower-order terms can be controlled by interpolation or by the conservation laws.

For the geometric estimates we will also need to control lower-order derivatives in the supremum norm. We define the \(C^2\) norm as
\[
\Ck[2]{\gamma} = \sup_{x\in\mathbb{T}}|\gamma(x)| + \sup_{x\in\mathbb{T}}|\partial_x\gamma(x)| + \sup_{x\in\mathbb{T}}|\partial_x^2\gamma(x)|.
\]
By the Sobolev embedding \(H^3(\mathbb{T})\hookrightarrow C^2(\mathbb{T})\) (in one dimension this embedding is continuous), there exists an absolute constant \(C>0\) such that
\[
\Ck[2]{\gamma} \le C \Hs{\gamma} \quad \text{for all } \gamma\in H^3(\mathbb{T}).
\]

In the lower regularity setting of Theorem 1.2 we will need Hölder spaces. For \(\sigma\in(0,1)\), define \(\Lipa[\sigma]\) as the space of functions \(f:\mathbb{T}\to\mathbb{R}^2\) such that
\[
\Na[\sigma]{f} := \sup_{x\in\mathbb{T}}|f(x)| + \sup_{x\neq y\in\mathbb{T}}\frac{|f(x)-f(y)|}{|x-y|^\sigma} < \infty.
\]

By Sobolev embedding, we have the continuous inclusion 
$H^2(\mathbb{T})\hookrightarrow \mathcal{C}^{1,\frac{1}{2}}.$ Here \(C^{1,\frac12}\) denotes the space of bounded functions whose first derivative
belongs to \(C^{\frac12}\). More generally, \(H^2(\mathbb T)\hookrightarrow
C^{1,\sigma}(\mathbb T)\) for every \(0<\sigma\le \frac12\). These embeddings will allow us to control the quantity \(F(\gamma)\) defined below, which involves differences of \(\gamma\) at nearby points.

To simplify notation, throughout the paper we will occasionally use the shorthand $$\diff{f} = f(x)-f(x-x').$$

\subsection{Kernel estimates}
We consider a radial kernel $K$, satisfying the conditions \ref{ass:smoothness}–\ref{ass:monotonicity} stated in the introduction. 

\begin{lemma} \label{lemma:intDerivades}
    If $K$ satisfies \ref{ass:smoothness}--\ref{ass:monotonicity}, then
    \begin{enumerate}
        \item $\displaystyle \lim_{r\to 0} r^{n+1}K^{(n)}(r)=0$ for $n=0,1,2$, \label{prop:Limits0}
        \item $\displaystyle \int_0^1 r^n\,|K^{(n)}(r)|\,dr < \infty$ for $n=1,2,3$. \label{prop:integrabilitatrK}
    \end{enumerate}
\end{lemma}

\begin{proof}
It is enough to prove the estimates near the origin, since the contribution away from
the origin is finite by the smoothness of \(K\). Hence, without loss of generality, we
assume that the monotonicity interval in \ref{ass:compKernel} is \((0,1]\).

Since \(|K|\) is monotone near the origin by \ref{ass:compKernel}, and
\(|K^{(m)}|\) is monotone for \(m=1,2,3\) by \ref{ass:monotonicity}, each \(K^{(m)}\)
has a fixed sign near the origin, up to the trivial case where it vanishes identically
on a subinterval. We shall use the analogous
observation for \(K'\), \(K''\), and \(K'''\) without further comment.

For any \(r\in(0,1]\), the monotonicity of \(|K|\) on \((0,1]\) gives
\[
    \int_0^r |K(t)|\,dt
    \ge
    \int_{r/2}^r |K(t)|\,dt
    \ge
    \frac r2 |K(r)|
    \ge 0 .
\]
Since \(\int_0^r |K(t)|\,dt\to0\) as \(r\to0\), we obtain
\[
    \lim_{r\to0} rK(r)=0,
\]
which proves \ref{prop:Limits0} for \(n=0\).

To treat \(K'\), integrate by parts on \([\varepsilon,1]\):
\[
    \int_\varepsilon^1 tK'(t)\,dt
    =
    tK(t)\Big|_\varepsilon^1
    -
    \int_\varepsilon^1 K(t)\,dt .
\]
Since \(\varepsilon K(\varepsilon)\to0\) and \(\int_0^1 K(t)\,dt\) converges, passing
\(\varepsilon\to0\) shows that \(\int_0^1 tK'(t)\,dt\) converges. As \(K'\) has fixed
sign near the origin, this convergence is absolute near the origin; hence
\[
    \int_0^1 t|K'(t)|\,dt<\infty.
\]
Using again the monotonicity of \(|K'|\),
\[
    \int_0^r t|K'(t)|\,dt
    \ge
    \int_{r/2}^r t|K'(t)|\,dt
    \ge
    \frac r2\int_{r/2}^r |K'(t)|\,dt
    \ge
    \frac{r^2}{4}|K'(r)|
    \ge 0 .
\]
The left-hand side tends to zero as \(r\to0\), and therefore
\[
    \lim_{r\to0} r^2K'(r)=0.
\]
This proves \ref{prop:Limits0} for \(n=1\), and also gives
\[
    \int_0^1 r|K'(r)|\,dr<\infty.
\]

The same argument, applied successively to \(K'\) and \(K''\), gives
\[
    \int_0^1 r^2|K''(r)|\,dr<\infty,
    \qquad
    \lim_{r\to0} r^3K''(r)=0,
\]
and
\[
    \int_0^1 r^3|K'''(r)|\,dr<\infty.
\]
Indeed, the integrations by parts are
\[
    \int_\varepsilon^1 t^2K''(t)\,dt
    =
    t^2K'(t)\Big|_\varepsilon^1
    -
    2\int_\varepsilon^1 tK'(t)\,dt,
\]
and
\[
    \int_\varepsilon^1 t^3K'''(t)\,dt
    =
    t^3K''(t)\Big|_\varepsilon^1
    -
    3\int_\varepsilon^1 t^2K''(t)\,dt.
\]
The boundary terms vanish because of the limits already proved, and the fixed-sign
property near the origin turns convergence into absolute convergence. The pointwise
limit for \(K''\) follows, as above, from
\[
    \int_0^r t^2|K''(t)|\,dt
    \ge
    \int_{r/2}^r t^2|K''(t)|\,dt
    \ge
    \frac{r^3}{8}|K''(r)|.
\]
This completes the proof.
\end{proof}

The following integral structure will appear repeatedly in what follows. For this reason we introduce a general bound, stated in the next lemma.

\begin{lemma} \label{lemma:canviVar}
     Let $j\geq n$ with $n=1,2,3$. If $K$ satisfies \ref{ass:smoothness}-\ref{ass:monotonicity} then for all $S>0$ we have 
     \begin{equation*}
         \int_0^S \nu^j |K^{(n)}(\nu)|d\nu \leq \int_0^1 \nu^j |K^{(n)}(\nu)| d\nu + \frac{|K^{(n)}(1)|}{j+1}(S^{j+1}-1)_+ \leq C_1 + C_2(S^{j+1}-1)_+,
     \end{equation*}
     where $(x)_+=\max\{x,0\},$ and the constants $C_1, C_2$ depend on $K$, $n$ and $j$.
\end{lemma}

\begin{proof}
    If $S\leq 1$ then $(S^{j+1}-1)_+=0$ so the result is clear. For $S\geq 1$ we split the integral into two parts and use the monotonicity assumption \ref{ass:monotonicity} to get
    \begin{align*}
        \int_0^S \nu^j |K^{(n)}(\nu)|d\nu &\leq \int_0^1 \nu^j |K^{(n)}(\nu)| d\nu + |K^{(n)}(1)| \int_1^S \nu^j  d\nu \\
        & \leq \int_0^1 \nu^j |K^{(n)}(\nu)| d\nu + \frac{|K^{(n)}(1)|}{j+1}(S^{j+1}-1).
    \end{align*}
    Finally, by Lemma \ref{lemma:intDerivades} we have that the first term is bounded.
\end{proof}

We shall also use that, since \(|K^{(n)}|\) is nonincreasing on \((0,\infty)\) for
\(n=1,2,3\), each \(K^{(n)}\) has a fixed sign on \((0,\infty)\), up to the harmless
case where it vanishes identically on a tail interval. Hence, for \(j\ge n\),
\[
\left|\int_0^S \nu^j K^{(n)}(\nu)\,d\nu\right|
=
\int_0^S \nu^j |K^{(n)}(\nu)|\,d\nu .
\]

\section{\texorpdfstring{Local well-posedness in $H^3$}{Local well-posedness in H3}}\label{sec:H3}
In this section we prove Theorem \ref{thm:H3}. As discussed in the introduction, the tangential component of the velocity does not affect the shape of the domain, only its parametrization. Exploiting this gauge freedom, we may take \(\lambda = 0\) in the contour dynamics equation \eqref{eq:CDEmod}, which fixes the contour patch equations as follows
\begin{equation}\label{eq:patch}
    \left\{\begin{aligned}
        \partial_t \gamma(x,t) &= -\int_\T K(|\gamma(x)-\gamma(x-x')|)(\partial_x \gamma(x)-\partial_x \gamma(x-x'))dx',\\ 
        \gamma(x,0)&= \gamma_0(x),
    \end{aligned}\right.
\end{equation}

where $\gamma_0\in H^3(\T)$ with $F(\gamma_0)(x,x')<\infty.$ We look at the time evolution of the $H^3$ norm $$\Hs{\gamma}^2=\Lp{\gamma}^2+\Lp{\partial_x^3\gamma }^2.$$

First, define $\Tilde{K}:[0,\infty)\to \R$ by $\Tilde{K}(u)=\int_0^{u} K(\sqrt{t})dt$ which is well-defined for $u\geq 0$ because the change of variables $t=r^2$ gives $\int_0^{u} K(\sqrt{t})dt=2\int_0^{\sqrt{u}} rK(r)dr$ and \ref{ass:integrability} guarantees convergence near zero. Then, by the Fundamental Theorem of Calculus, $\Tilde{K}'(u)=K(\sqrt{u})$ for $u>0.$ In particular, $K(r)=\Tilde{K}'(r^2)$ . Using the symmetry $x\leftrightarrow x-x'$ (which leaves the torus invariant) we obtain
\begin{equation}
\begin{aligned}
    -\int_\T \gamma(x) \cdot \gamma_t(x)dx &= \intTT \gamma(x) \cdot (\partial_x\gamma(x)-\partial_x\gamma(x-x'))K(|\diff{\gamma}|)dx'dx \\
    &= -\int_\T\int_\T \gamma(x-x') \cdot (\partial_x\gamma(x)-\partial_x\gamma(x-x'))K(|\diff{\gamma}|)dx'dx \\
    &= \frac{1}{2}\int_\T\int_\T \diff{\gamma}\cdot(\partial_x\gamma(x)-\partial_x\gamma(x-x'))K(|\diff{\gamma}|)dx'dx \\
    &= \frac{1}{4} \int_\T\int_\T \partial_x [|\gamma(x)-\gamma(x-x')|^2]K(|\gamma(x)-\gamma(x-x')|)dx'dx \\
    &= \frac{1}{4} \int_\T\int_\T \partial_x \Tilde{K}(|\gamma(x)-\gamma(x-x')|^2) dx'dx = 0.
\end{aligned}\label{eq:gammaL2}
\end{equation}

The final equality follows from Fubini's theorem and the fact that $\gamma$ is a closed curve, so the integral of a total derivative over $\mathbb{T}$ vanishes. Therefore $\partial_t\Lp{\gamma}(t)=0.$ 

We control now the quantity $$-\int_\T \partial_x^3\gamma(x)\cdot \partial_x^3\gamma_t(x)dx= I_1+ I_2+ I_3+ I_4,$$ where 
\begin{align*}
    I_1&=\int_\T\int_\T\partial_x^3\gamma(x)\cdot(\partial_x^4\gamma(x)-\partial_x^4\gamma(x-x'))K(|\gamma(x)-\gamma(x-x')|)dx'dx,\\ 
    I_2&=3\int_\T\int_\T\partial_x^3\gamma(x)\cdot(\partial_x^3\gamma(x)-\partial_x^3\gamma(x-x'))\partial_x[K(|\gamma(x)-\gamma(x-x')|)]dx'dx,\\
    I_3&=3\int_\T\int_\T\partial_x^3\gamma(x)\cdot(\partial_x^2\gamma(x)-\partial_x^2\gamma(x-x'))\partial_x^2[K(|\gamma(x)-\gamma(x-x')|)]dx'dx,\\
    I_4&=\int_\T\int_\T\partial_x^3\gamma(x)\cdot(\partial_x\gamma(x)-\partial_x\gamma(x-x'))\partial_x^3[K(|\gamma(x)-\gamma(x-x')|)]dx'dx.
\end{align*}

For $I_1$ we proceed as in \eqref{eq:gammaL2} to obtain
\begin{align*}
    I_1&=\frac{1}{2}\int_\T\int_\T (\partial_x^3\gamma(x)-\partial_ x^3\gamma(x-x')) \cdot (\partial_x^4\gamma(x)-\partial_x^4\gamma(x-x'))K(|\diff{\gamma}|)dx'dx\\
    &=\frac{1}{4}\int_\T\int_\T \partial_x[|\partial_x^3\gamma(x)-\partial_x^3\gamma(x-x')|^2]K(|\diff{\gamma}|)dx'dx\\
    &=-\frac{1}{4} \int_\T\int_\T |\partial_x^3\gamma(x)-\partial_x^3\gamma(x-x')|^2 \partial_x[K(|\diff{\gamma}|)]dx'dx\\
    &= -\frac{1}{4} \int_\T\int_\T |\diff{\partial_x^3\gamma}|^2\; \frac{\diff{\gamma}\cdot \diff{\partial_x\gamma}}{|\diff{\gamma}|}\;  K'(|\diff{\gamma}|) dx' dx.
\end{align*}

Due to the inequality $|\partial_x\gamma(x)-\partial_x\gamma(x-x')||x'|^{-1}\leq \Ck{\gamma}$, and the monotonicity assumption, it follows that 
\begin{align*}
    |I_1|&\leq C \Ck{\gamma}  \int_\T\int_\T |\partial_x^3\gamma(x)-\partial_x^3\gamma(x-x')|^2 |x'| |K'(|\gamma(x)-\gamma(x-x')|)| dx'dx \\
    &= C \Ck{\gamma} \int_\T\int_\T |\partial_x^3\gamma(x)-\partial_x^3\gamma(x-x')|^2 |x'| \left|K'\left(\frac{|x'|}{F(\gamma)}\right)\right|dx'dx\\
    &\leq C \Ck{\gamma} \int_\T\int_\T |\partial_x^3\gamma(x)-\partial_x^3\gamma(x-x')|^2 |x'| \left|K'\left(\frac{|x'|}{\Np{F(\gamma)}}\right)\right|dx'dx\\
    &= C \Ck{\gamma} \int_\T |x'| \left|K'\left(\frac{|x'|}{\Np{F(\gamma)}}\right)\right| \int_\T |\partial_x^3\gamma(x)-\partial_x^3\gamma(x-x')|^2 dxdx'\\
    &\leq  C \Ck{\gamma} \int_\T |x'| \left|K'\left(\frac{|x'|}{\Np{F(\gamma)}}\right)\right| \int_\T |\partial_x^3\gamma(x)|^2+|\partial_x^3\gamma(x-x')|^2 dxdx'\\
    &\leq  C \Np{F(\gamma)} \Ck{\gamma} \Lp{\partial_x^3\gamma}^2 \left|\int_\T \frac{|x'|}{\Np{F(\gamma)}} K'\left(\frac{|x'|}{\Np{F(\gamma)}}\right)dx' \right|\\
    &\leq C \Np{F(\gamma)}^2 \Ck{\gamma} \Lp{\partial_x^3\gamma}^2 \left|\int_0^{\frac{\pi}{\Np{F(\gamma)}}} \nu K'\left(\nu\right)d\nu \right|.
\end{align*}

Using lemma \ref{lemma:canviVar} with $S=\frac{\pi}{\Np{F(\gamma)}}$  finally gives $$|I_1|\leq  C (\Np{F(\gamma)}^2 + 1 )\Ck{\gamma} \Lp{\partial_x^3\gamma}^2. $$

By proceeding as before it is easy to see that $I_2=-6I_1$, so the same bound holds for $I_2$. We now estimate $\frac{1}{3}I_3=J_1+J_2+J_3+J_4$ where 

{\allowdisplaybreaks
\begin{align*}
    J_1 &= \intTT \partial_x^3\gamma(x) \cdot \diff{\partial_x^2\gamma} 
            K''(|\diff{\gamma}|) \, \frac{(\diff{\gamma} \cdot \diff{\partial_x\gamma})^2}{|\diff{\gamma}|^2} \, dx' \, dx,\\[4pt]
    J_2 &= \intTT \partial_x^3\gamma(x) \cdot \diff{\partial_x^2\gamma} \,
            K'(|\diff{\gamma}|) \, \frac{\diff{\gamma} \cdot \diff{\partial_x^2\gamma}}{|\diff{\gamma}|} \, dx' \, dx,\\[4pt]
    J_3 &= \intTT \partial_x^3\gamma(x) \cdot \diff{\partial_x^2\gamma}  \,
            K'(|\diff{\gamma}|) \, \frac{|\diff{\partial_x\gamma}|^2}{|\diff{\gamma}|} \, dx' \, dx,\\[4pt]
    J_4 &= -\intTT \partial_x^3\gamma(x) \cdot \diff{\partial_x^2\gamma} \,
            K'(|\diff{\gamma}|) \, \frac{(\diff{\gamma} \cdot \diff{\partial_x\gamma})^2}{|\diff{\gamma}|^3} \, dx' \, dx .
\end{align*}
}

    Notice that by the mean value inequality
    \begin{equation}\label{eq:MVTgammaH3}
        \diff{\partial_x^2\gamma}=\partial_ x^2\gamma(x)-\partial_x^2\gamma(x-x')=x'\int_0^1\partial_x^3 \gamma(x+(s-1)x')ds,
    \end{equation} thus
    \begin{align*}
        |&J_1|\leq \Ck{\gamma}^2\int_\T\int_\T |x'|^3 |K''(|\diff{\gamma}|)| \int_0^1 |\partial_x^3\gamma(x)||\partial_x^3 \gamma(x+(s-1)x')|ds dx'dx \\
        & \leq \Ck{\gamma}^2\int_\T\int_\T |x'|^3 |K''(|\diff{\gamma}|)| \int_0^1 |\partial_x^3\gamma(x)|^2+|\partial_x^3 \gamma(x+(s-1)x')|^2ds dx'dx \\
        &\leq 2\Ck{\gamma}^2 \Lp{\partial_x^3\gamma}^2\int_\T\int_\T |x'|^3 |K''(|\gamma(x)-\gamma(x-x')|)| dx'dx \\
        &= 2\Ck{\gamma}^2 \Lp{\partial_x^3\gamma}^2\int_\T\int_\T |x'|^3 \left|K''\left(\frac{|x'|}{F(\gamma)}\right)\right| dx'dx \\
        &\leq 2\Ck{\gamma}^2 \Lp{\partial_x^3\gamma}^2 \Np{F(\gamma)}^3\left|\int_\T\int_\T \frac{|x'|^3}{\Np{F(\gamma)}^3} K''\left(\frac{|x'|}{\Np{F(\gamma)}}\right) dx'dx \right| \\
        & \leq C \Ck{\gamma}^2 \Lp{\partial_x^3\gamma}^2 \Np{F(\gamma)}^4\left|\int_\T\int_0^{\frac{\pi}{\Np{F(\gamma)}}} \nu^3 K''(\nu) d\nu dx \right|,
    \end{align*}
    and by Lemma \ref{lemma:canviVar} 
    $$|J_1|\leq C \Ck{\gamma}^2 \Lp{\partial_x^3\gamma}^2 (\Np{F(\gamma)}^4+1).$$

    And proceeding in the same way
    \begin{align*}
        |J_2|&\leq \int_\T\int_\T |\diff{\partial_x^2\gamma}| |K'(|\diff{\gamma}|)|x'|\int_0^1 |\partial_x^3\gamma(x)||\partial_x^3\gamma(x+(s-1)x')|ds dx'dx \\
        &\leq C \Lp{\partial_x^3\gamma}^2 \int_\T \int_\T (|\partial_x^2\gamma(x)|+|\partial_x^2\gamma(x-x')|) |K'(|\gamma(x)-\gamma(x-x')|)|x'| dx'dx\\
        &\leq C \Lp{\partial_x^3\gamma}^2 \Np{F(\gamma)}^2 \Ck{\gamma} \left|\int_\T\int_0^{\frac{\pi}{\Np{F(\gamma)}}} \nu K'(\nu) d\nu dx\right|\\
        &\leq C \Lp{\partial_x^3\gamma}^2 \Ck{\gamma} (\Np{F(\gamma)}^2 + 1),
    \end{align*}

    and 
    \begin{align*}
        |J_3|&\leq \Ck{\gamma}^2\Np{F(\gamma)} \int_\T\int_\T |K'(|\diff{\gamma}|)|x'|^2\int_0^1 |\partial_x^3\gamma(x)||\partial_x^3\gamma(x+(s-1)x')|ds dx'dx
        \\ &\leq  C \Lp{\partial_x^3\gamma}^2 \Np{F(\gamma)}^3 \Ck{\gamma}^2\left|\int_\T\int_0^{\frac{\pi}{\Np{F(\gamma)}}} \nu^2 K'(\nu) d\nu dx\right|
        \\ &\leq C \Lp{\partial_x^3\gamma}^2 \Ck{\gamma}^2 (\Np{F(\gamma)}^3+1),
    \end{align*}

    for $J_4$ one gets the same bound as for $J_3,$ so one gets 
    \begin{align*}
        |I_3|\leq& C \Ck{\gamma}\Lp{\partial_x^3\gamma}^2 \bigg( (1+\Ck{\gamma}\Np{F(\gamma)}+\Ck{\gamma}\Np{F(\gamma)}^2)\\ &\hspace{2cm}\cdot(\Np{F(\gamma)}^2+ 1+\Ck{\gamma}) \bigg).
    \end{align*}   

Finally we write $I_4=\sum_{i=5}^{13} J_i$ where

{\allowdisplaybreaks
\thinmuskip=2mu
\medmuskip=2.5mu plus 1mu minus 1mu
\thickmuskip=4mu plus 2mu minus 1mu
\jot=2pt
\begin{align*}
    J_5 &= \intTT \partial_x^3\gamma(x) \cdot \diff{\partial_x\gamma} \,
            K'''(|\diff{\gamma}|) \, \frac{(\diff{\gamma} \cdot \diff{\partial_x\gamma})^3}{|\diff{\gamma}|^3} \, dx' dx,\\[4pt]
    J_6 &= 3  \intTT  \partial_x^3\gamma(x) \cdot \diff{\partial_x\gamma}  \,
            K''(|\diff{\gamma}|) \, \frac{(\diff{\gamma} \cdot \diff{\partial_x\gamma})\,|\diff{\partial_x\gamma}|^2}{|\diff{\gamma}|^2} \, dx' dx,\\[4pt]
    J_7 &= 3 \intTT  \partial_x^3\gamma(x) \cdot \diff{\partial_x\gamma}
            K''(|\diff{\gamma}|) \frac{(\diff{\gamma} \cdot \diff{\partial_x\gamma})(\diff{\gamma} \cdot \diff{\partial_x^2\gamma})}{|\diff{\gamma}|^2} dx' dx,\\[4pt]
    J_8 &= -3  \intTT  \partial_x^3\gamma(x) \cdot \diff{\partial_x\gamma} \,
            K''(|\diff{\gamma}|) \, \frac{(\diff{\gamma} \cdot \diff{\partial_x\gamma})^3}{|\diff{\gamma}|^4} \, dx' dx,\\[4pt]
    J_9 &= 3 \intTT  \partial_x^3\gamma(x) \cdot \diff{\partial_x\gamma} \,
            K'(|\diff{\gamma}|) \, \frac{\diff{\partial_x\gamma} \cdot \diff{\partial_x^2\gamma}}{|\diff{\gamma}|} \, dx' dx,\\[4pt]
    J_{10} &= -3 \intTT  \partial_x^3\gamma(x) \cdot \diff{\partial_x\gamma} \,
            K'(|\diff{\gamma}|) \, \frac{|\diff{\partial_x\gamma}|^2\,(\diff{\gamma} \cdot \diff{\partial_x\gamma})}{|\diff{\gamma}|^3} \, dx' dx,\\[4pt]
    J_{11} &=  \intTT  \partial_x^3\gamma(x) \cdot \diff{\partial_x\gamma} \,
            K'(|\diff{\gamma}|) \, \frac{\diff{\gamma} \cdot \diff{\partial_x^3\gamma}}{|\diff{\gamma}|} \, dx' dx,\\[4pt]
   J_{12} &= -3 \int\limits_{\mathbb T}\mkern-8mu\int\limits_{\mathbb T} \partial_x^3\gamma(x)\cdot\diff{\partial_x\gamma}
            K'(|\diff{\gamma}|) \frac{(\diff{\gamma}\cdot\diff{\partial_x\gamma})(\diff{\gamma}\cdot\diff{\partial_x^2\gamma})}{|\diff{\gamma}|^3} dx' dx,\\[4pt]
    J_{13} &= 3  \intTT  \partial_x^3\gamma(x) \cdot \diff{\partial_x\gamma}  \,
            K'(|\diff{\gamma}|) \, \frac{(\diff{\gamma} \cdot \diff{\partial_x\gamma})^3}{|\diff{\gamma}|^5} \, dx' dx .
\end{align*}
}

    By proceeding as in the previous cases we have
    \begin{align*}
        |J_5|&\leq C \Ck{\gamma}^3 \int_\T |x'|^3\left| K^{(3)}\left( \frac{|x'|}{F(\gamma)}\right) \right| \int_\T |\partial_x^3\gamma(x)||\partial_x\gamma(x)-\partial_x\gamma(x-x')| dxdx'\\
        &\leq C \Ck{\gamma}^3 \Lp{\partial_x\gamma }^2\Lp{\partial_x^3\gamma }^2\Np{F(\gamma)}^4 \left|\int_0^{\frac{\pi}{\Np{F(\gamma)}}} \nu^3 K^{(3)}(\nu) d\nu \right|, 
    \end{align*}

    \begin{align*}
        |J_6|&\leq C \Ck{\gamma}^3 \Np{F(\gamma)}\int_\T |x'|^2\left| K''\left( \frac{|x'|}{F(\gamma)}\right) \right| \int_\T |\partial_x^3\gamma(x)||\partial_x\gamma(x)-\partial_x\gamma(x-x')| dxdx'\\
        &\leq C \Ck{\gamma}^3 \Lp{\partial_x\gamma }^2\Lp{\partial_x^3\gamma }^2\Np{F(\gamma)}^4 \left|\int_0^{\frac{\pi}{\Np{F(\gamma)}}} \nu^2 K''(\nu) d\nu \right|, 
    \end{align*}
    the same bound holds for $|J_8|$.
    \begin{align*}
        |J_7|&\leq C \Ck{\gamma}^2 \int_\T\int_\T |x'|^3\left|K''\left(\frac{|x'|}{F(\gamma)}\right)\right| \int_0^1 |\partial_x^3\gamma(x)||\partial_x^3\gamma(x+(s-1)x')|dsdx'dx\\
        &\leq C \Ck{\gamma}^2 \Lp{\partial_x^3\gamma}^2\Np{F(\gamma)}^4 \left|\int_\T\int_0^{\frac{\pi}{\Np{F(\gamma)}}} \nu^3 K''(\nu) d\nu dx \right|,
    \end{align*}

     \begin{align*}
        |J_9|&\leq C \Ck{\gamma}^2 \Np{F(\gamma)} \int_\T\int_\T |x'|^2\left|K'\left(\frac{|x'|}{F(\gamma)}\right)\right| \int_0^1 |\partial_x^3\gamma(x)||\partial_x^3\gamma(x+(s-1)x')|dsdx'dx\\
        &\leq C \Ck{\gamma}^2 \Lp{\partial_x^3\gamma}^2\Np{F(\gamma)}^4 \left|\int_\T\int_0^{\frac{\pi}{\Np{F(\gamma)}}} \nu^2 K'(\nu) d\nu dx \right|,
    \end{align*}

    the same bound holds for $|J_{12}|.$

    \begin{align*}
        |J_{10}|&\leq C \Ck{\gamma}^3 \Np{F(\gamma)}^2\int_\T |x'|\left| K'\left( \frac{|x'|}{F(\gamma)}\right) \right|  \int_\T |\partial_x^3\gamma(x)||\partial_x\gamma(x)-\partial_x\gamma(x-x')| dxdx'\\
        &\leq C \Ck{\gamma}^3 \Lp{\partial_x\gamma }^2\Lp{\partial_x^3\gamma }^2\Np{F(\gamma)}^4 \left|\int_0^{\frac{\pi}{\Np{F(\gamma)}}} \nu K'(\nu) d\nu \right|, 
    \end{align*}

    the same bound holds for $|J_{13}|.$
    
    \begin{align*}
        |J_{11}|&\leq \Ck{\gamma} \int_\T |x'| \left| K'\left(\frac{|x'|}{\Np{F(\gamma)}}\right)\right|\int_\T |\partial_x^3\gamma(x)|^2+|\partial_x^3\gamma(x)||\partial_x^3\gamma(x-x')| dxdx' \\
        &\leq C \Ck{\gamma} \Lp{\partial_x^3\gamma}^2 \Np{F(\gamma)}^2 \left|\int_0^{\frac{\pi}{\Np{F(\gamma)}}} \nu K'(\nu) d\nu\right|.
    \end{align*}

   So 
   \begin{align*}
        |I_4|
        &\leq C \Ck{\gamma}\Lp{\partial_x^3\gamma}^2
        \Big(\Np{F(\gamma)}^4\big(\Ck{\gamma}^2 \Lp{\partial_x\gamma}^2+ \Ck{\gamma}
        \big) 
        \\&\qquad + \Np{F(\gamma)}^2\big(\Ck{\gamma}^2 \Lp{\partial_x\gamma}^2
        + 1
        \big)
        + \Np{F(\gamma)}
        \big(
        \Ck{\gamma}^2 \Lp{\partial_x\gamma}^2
        + \Ck{\gamma}
        \big) \\
        &\qquad
        + \big(
        \Ck{\gamma}^2 \Lp{\partial_x\gamma}^2
        + \Ck{\gamma}
        + 1
        \big)
        \Big).
    \end{align*}
  
    Combining all of the above expressions one gets that $$\frac{d}{dt} \Lp{\partial_x^3 \gamma}^2(t) \leq C (1+\Np{F(\gamma)}^4(t))\Ck{\gamma}^3 (t) \Hs{\gamma}^2(t),$$ and by definition of the Sobolev space norm also 
    $$\frac{d}{dt} \Hs{\gamma}(t) \leq C (1+\Np{F(\gamma)}(t))^4\Ck{\gamma}^3 (t) \Hs{\gamma}(t).$$
    Finally, using Sobolev inequalities, we obtain
    \begin{equation}\label{eq:evolucioSobolevNormH3}
    \frac{d}{dt} \Hs{\gamma}(t) \leq C (1+\Np{F(\gamma)}(t))^4\Hs{\gamma}^4(t).\end{equation}
  
   Following \cite{gancedo2008SQG}, we regularize equation \eqref{eq:patch} in order to apply energy methods (see \cite{MajdaBertozzi} for a comprehensive treatment):
    \begin{equation}\label{eq:patchRegular}
    \scalebox{0.99}{$
    \left\{\begin{aligned}
        \partial_t \gamma^\varepsilon(x,t) &= - \phi_\varepsilon*\int_\T K(|\gamma^\varepsilon(x)-\gamma^\varepsilon(x-x')|)(\partial_x(\phi_\varepsilon* \gamma^\varepsilon(x)-\phi_\varepsilon*\gamma^\varepsilon(x-x')))dx',\\ 
        \gamma^\varepsilon(x,0)&= \gamma_0(x),
    \end{aligned}\right.$}
    \end{equation}
    
    where $\phi_{\varepsilon}$ is a regular approximation to the identity. If the inequality \eqref{eq:arcChord}
   holds initially, due to the properties of the regular approximations to the identity, we get
    a Picard system as follows
    \begin{align*}
    \gamma^\varepsilon_t(x,t)&=G^\varepsilon(\gamma^{\varepsilon}(x,t)),\\
    \gamma^\varepsilon(x,0)&=\gamma_0(x),
    \end{align*}
    where $G^{\varepsilon}$ is Lipschitz. Therefore, for any $\varepsilon>0$, we obtain a time of existence $t_\varepsilon$
    where \eqref{eq:arcChord} is fulfilled. To obtain a time of existence for \eqref{eq:patchRegular} that is independent of $\varepsilon$, one needs energy estimates with bounds uniform in $\varepsilon$. Passing then to the limit $\varepsilon \to 0$ yields solutions of the original equation. In this particular case, we obtain
    \begin{equation*}
    \frac{d}{dt}\Hs{\gamma^\varepsilon}(t)\leq C (1+\Np{F(\gamma^\varepsilon)}^{4}(t))\Hs{\gamma^{\varepsilon} }^4(t),
    \end{equation*}
    and it may happen that $\Np{F(\gamma^\varepsilon)}\rightarrow \infty$ as $\varepsilon\to 0$. The energy estimate depends on $\varepsilon$, and the argument breaks down. In particular, even if the initial data satisfy \eqref{eq:arcChord}, we cannot guarantee the existence of a time $t>0$ independent of $\varepsilon$ for which \eqref{eq:arcChord} continues to hold. At this stage of the proof, the limiting system as $\varepsilon \to 0$ is not yet well-posed, since the Lipschitz constant of $G^\varepsilon$ may diverge as $\varepsilon \to 0$.
    
    To overcome this difficulty, we consider the evolution of the quantity $\Np{F(\gamma)}.$    
    
    Taking $p>2$ it follows 
    \begin{align*}
        \frac{d}{dt} \Np[p]{F(\gamma)}^p&= \frac{d}{dt} \int_\T\int_\T \left( \frac{|x'|}{|\gamma(x,t)-\gamma(x-x',t)|}\right)^p dxdx'\\
        &\leq p \int_\T\int_\T \left( \frac{|x'|}{|\gamma(x,t)-\gamma(x-x',t)|}\right)^{p+1}\frac{|\gamma_t(x,t)-\gamma_t(x-x',t)|}{|x'|} dxdx'.
    \end{align*}
    From \eqref{eq:patch}, we have 
    \begin{align*}
    \gamma_t(x)-\gamma_t(x-x')
    &= \int_{\T} K\bigl(|\diff[\xi]{\gamma}|\bigr)\,\diff[\xi]{\partial_x\gamma}\,d\xi 
       - \int_{\T} K\bigl(|\diff[\xi][x-x']{\gamma}|\bigr)\,\diff[\xi][x-x']{\partial_x\gamma}\,d\xi \\
    &= \int_{\T} \Bigl[K\bigl(|\diff[\xi]{\gamma}|\bigr)-K\bigl(|\diff[\xi][x-x']{\gamma}|\bigr)\Bigr]\,\diff[\xi]{\partial_x\gamma}\,d\xi \\
    &\quad - \int_{\T} K\bigl(|\diff[\xi][x-x']{\gamma}|\bigr)\,\Bigl[\diff[\xi]{\partial_x\gamma} - \diff[\xi][x-x']{\partial_x\gamma}\Bigr]\,d\xi \\
    &= I_5 + I_6.
\end{align*}

    Notice that 
    \begin{equation}\label{eq:diferenciaKernelsRTI}
        |K(r)-K(r')| \leq  |r-r'|\int_0^1 |K'(s r + (1-s) r')|ds\leq |r-r'| (|K'(r)|+|K'(r')|),
    \end{equation}    
    thus, taking $r=|\gamma(x)-\gamma(x-\xi)|$, $r'=|\gamma(x-x')-\gamma(x-x'-\xi)|$ and using reverse triangular inequality ($||x|-|y||\leq |x-y| $) one gets 
    {\allowdisplaybreaks
    \begin{align*}
        |I_5|&\leq C\Ck{\gamma} \int_\T |\xi|\big|K(|\gamma(x)-\gamma(x-\xi)|)-K(|\gamma(x-x')-\gamma(x-x'-\xi)|)\big| d\xi \\       
        &\leq C\Ck{\gamma} \int_\T |\xi|(|\gamma(x)-\gamma(x-x')| + |\gamma(x-\xi)-\gamma(x-x'-\xi))| \\&\cdot [|K'(|\gamma(x)-\gamma(x-\xi)|)|+|K'(|\gamma(x-x')-\gamma(x-x'-\xi)|)|]d\xi\\
        &\leq C\Ck{\gamma}^2 |x'| \int_\T |\xi| [|K'(|\diff[\xi]{\gamma}|)|+|K'(|\diff[\xi][x-x']{\gamma}|)|]d\xi \\
        &\leq C\Ck{\gamma}^2 |x'| \Np{F(\gamma)} \int_\T \frac{|\xi|}{\Np{F(\gamma)}} \left|K'\left(\frac{|\xi|}{\Np{F(\gamma)}}\right)\right|d\xi \\
        &\leq C\Ck{\gamma}^2 |x'| \Np{F(\gamma)}^2 \left|\int_0^{\frac{\pi}{\Np{F(\gamma)}}} \nu K'(\nu) d\nu \right| \\&\leq C\Ck{\gamma}^2 |x'| (\Np{F(\gamma)}^2 + 1).
    \end{align*}}

    Finally, for $I_6$ we have 
    \begin{align*}
        |I_6|&\leq \int_\T |K(|\gamma(x-x')-\gamma(x-x'-\xi)|)| \\ & \cdot  [|\partial_x\gamma(x)-\partial_x\gamma(x-x')|+|\partial_x\gamma(x-x'-\xi)-\partial_x\gamma(x-\xi)|]d\xi\\
        &\leq C \Ck{\gamma} |x'| \int_\T |K(|\gamma(x-x')-\gamma(x-x'-\xi)|)| d\xi.
    \end{align*}

    Let \(d = |\gamma(x-x')-\gamma(x-x'-\xi)|\). From the chord‑arc condition and the mean-value inequality,
    \[
    \frac{1}{\Np{F(\gamma)}}|\xi| \le d \le \Np{\partial_x\gamma}|\xi|,\qquad \text{ and } \Np{F(\gamma)}\Np{\partial_x\gamma}\geq 1.
    \]
    
    We split the \(\xi\)-integral into three regions based on the size of \(|\xi|\):
    \begin{equation}\label{eq:splitKernel3regions}
        \int_\T |K(|\gamma(x-x')-\gamma(x-x'-\xi)|)| d\xi = I + II + III.
    \end{equation}
    
    Empty integration regions are understood to give zero contribution. Then
    
    \begin{align*}
        I&=\int_{|\xi| \le \delta/\Np{\partial_x\gamma}} |K(d)|d\xi \le \int_{|\xi|\le\delta/\Np{\partial_x\gamma}} |K\left(\frac{|\xi|}{\Np{F(\gamma)}}\right)|\,d\xi\\& = 2\int_0^{\delta/\Np{\partial_x\gamma}} |K\left(\frac{|\xi|}{\Np{F(\gamma)}}\right)|\,d\xi = 2 \Np{F(\gamma)} \int_0^{\frac{\delta}{\Np{F(\gamma)} \Np{\partial_x\gamma}}} |K(\nu)|\,d\nu\\& \le 2\Np{F(\gamma)}\int_0^{\delta} |K(\nu)|\,d\nu = C \Np{F(\gamma)},
    \end{align*}
    
    since in this region \(d \le \Np{\partial_x\gamma}|\xi| \le \delta\), so \(d\in(0,\delta]\) and therefore we can use monotonicity of $|K|$ up to $\delta$.   

    \begin{align*}
    II&=
\int_{\delta/\|\partial_x\gamma\|_{L^\infty}}^{\delta\|F(\gamma)\|_{L^\infty}} |K(d)|\,d\xi 
\\&\le \int_{\delta/\|\partial_x\gamma\|_{L^\infty}}^{\delta\|F(\gamma)\|_{L^\infty}} \left|K\left(\frac{|\xi|}{\|F(\gamma)\|_{L^\infty}}\right)\right| d\xi + C\int_{\delta/\|\partial_x\gamma\|_{L^\infty}}^{\delta\|F(\gamma)\|_{L^\infty}} \bigl(1 + \|\partial_x\gamma\|_{L^\infty}^M |\xi|^M\bigr)\,d\xi.
\end{align*}

The first term is treated like in $I$
\begin{align*}
\int_{\delta/\|\partial_x\gamma\|_{L^\infty}}^{\delta\|F(\gamma)\|_{L^\infty}} \left|K\left(\frac{|\xi|}{\|F(\gamma)\|_{L^\infty}}\right)\right| d\xi
&\leq 2\|F(\gamma)\|_{L^\infty} \int_{\frac{\delta}{\|F(\gamma)\|_{L^\infty}\|\partial_x\gamma\|_{L^\infty}}}^{\delta} |K(\nu)|\,d\nu
\\&\le 2\|F(\gamma)\|_{L^\infty} \int_0^{\delta} |K(\nu)|\,d\nu = C\|F(\gamma)\|_{L^\infty}.
\end{align*}

The second term is bounded by
\[
C\int_0^{\pi} d\xi + C \|\partial_x\gamma\|_{L^\infty}^M \int_0^{\pi} |\xi|^M d\xi \le C(1 + \|\partial_x\gamma\|_{L^\infty}^M),
\]
since the integration interval is contained in \([0,\pi]\).

Finally,
\begin{align*}
    III&=\int_{|\xi|>\delta\|F(\gamma)\|_{L^\infty}} |K(d)|\,d\xi 
    \le C\int_{\delta\|F(\gamma)\|_{L^\infty}}^{\pi} \bigl(1 + \|\partial_x\gamma\|_{L^\infty}^M |\xi|^M\bigr)\,d\xi \\
    &\le C\int_0^{\pi} d\xi + C \|\partial_x\gamma\|_{L^\infty}^M \int_0^{\pi} |\xi|^M d\xi \le C(1 + \|\partial_x\gamma\|_{L^\infty}^M),
\end{align*}
since \(|\xi|/\|F(\gamma)\|_{L^\infty} > \delta\), so \(d \ge |\xi|/\|F(\gamma)\|_{L^\infty} > \delta\) and the polynomial growth can be applied.
    
    Therefore
    \begin{equation} \label{eq:controlIntKernelDif}
         \intT |K(|\gamma(x-x')-\gamma(x-x'-\xi)|)|\,d\xi \le C\left( \Np{F(\gamma)} + 1 + \Np{\partial_x\gamma}^M \right),
    \end{equation}  

   and 
   $$|I_6|\leq C \Ck{\gamma} |x'| \bigl( \Np{F(\gamma)} + 1 + \|\partial_x\gamma\|_{L^\infty}^M \bigr). $$
   
The estimates for \(I_5\) and \(I_6\) obtained above yield
\begin{align*}
\frac{|\gamma_t(x)-\gamma_t(x-x')|}{|x'|}
&\le C\Bigl( \|\gamma\|_{C^2}^2(t)\|F\|_{L^\infty}^2(t) + \|\gamma\|_{C^2}(t)\|F\|_{L^\infty}(t) + \|\gamma\|_{C^2}^{M+1}(t) \Bigr)\\
&\eqqcolon C\mathcal{P}_0(\Ck{\gamma}(t),\Np{F(\gamma)}(t)),
\end{align*}
where \(M\) is the exponent from the polynomial growth condition \ref{ass:compKernel} and $\mathcal{P}_0$ a polynomial. Consequently,
\begin{align*}
\frac{d}{dt}\|F\|_{L^p}^p(t) &\le pC\,\mathcal{P}_0(\Ck{\gamma}(t),\Np{F(\gamma)}(t)) \int_{\mathbb{T}^2} F^{p+1}(x,x',t)\,dx\,dx'\\
&\le pC\,\mathcal{P}_0(\Ck{\gamma}(t),\Np{F(\gamma)}(t))\|F\|_{L^\infty}(t)  \|F\|_{L^p}^p(t).
\end{align*}

Thus,
\[
\frac{d}{dt}\|F\|_{L^p}(t) \le C\,\mathcal{P}_0(\Ck{\gamma}(t),\Np{F(\gamma)}(t)) \|F\|_{L^\infty}(t) \|F\|_{L^p}(t).
\]

Integrating in time and then letting \(p\to\infty\) (see \cite{gancedo2008SQG} for the limiting argument) we obtain
\begin{multline*}
    \|F\|_{L^\infty}(t+h) \le \|F\|_{L^\infty}(t) \exp\Bigl( C\int_t^{t+h} \bigl( \|\gamma\|_{C^2}^2(s)\|F\|_{L^\infty}^3(s) + \|\gamma\|_{C^2}(s)\|F\|_{L^\infty}^2(s) \\+ \|\gamma\|_{C^2}^{M+1}(s)\|F\|_{L^\infty}(s) \bigr) ds \Bigr).
\end{multline*}

Differentiating at \(h=0\) yields
\[
\frac{d}{dt}\|F\|_{L^\infty}(t) \le C\Bigl( \|\gamma\|_{C^2}^2(t)\|F\|_{L^\infty}^4(t) + \|\gamma\|_{C^2}(t)\|F\|_{L^\infty}^3(t) + \|\gamma\|_{C^2}^{M+1}(t)\|F\|_{L^\infty}^2(t) \Bigr).
\]

Applying Sobolev inequalities \(\|\gamma\|_{C^2}(t)\le C\|\gamma\|_{H^3}(t)\) we obtain
\[
\frac{d}{dt}\|F\|_{L^\infty}(t) \le C\Bigl( \|\gamma\|_{H^3}^2(t)\|F\|_{L^\infty}^4(t) + \|\gamma\|_{H^3}(t)\|F\|_{L^\infty}^3(t) + \|\gamma\|_{H^3}^{M+1}(t)\|F\|_{L^\infty}^2(t) \Bigr).
\]

All terms on the right-hand side are of polynomial order in \(\|\gamma\|_{H^3}\) and \(\|F\|_{L^\infty}\). Adding this estimate to the energy inequality \eqref{eq:evolucioSobolevNormH3} we obtain
\[
\frac{d}{dt}\bigl( \|\gamma\|_{H^3}(t) + \|F\|_{L^\infty}(t) \bigr) \le C \bigl( \|\gamma\|_{H^3}(t) + \|F\|_{L^\infty}(t) \bigr)^{M'},
\]
for some sufficiently large \(M'\) (for instance, \(M' = \max\{8, M+4\}\)). Integrating the differential inequality we obtain for sufficiently small \(t\)
\begin{equation}\label{eq:integralLocalExistenceH3}
\scalebox{1.1}{$
    \Hs{\gamma}(t) + \Np{F(\gamma)}(t) \le \frac{\Hs{\gamma_0} + \Np{F(\gamma_0)}}{\bigl(1 - tC(M'-1)\bigl(\Hs{\gamma_0} + \Np{F(\gamma_0)}\bigr)^{M'-1}\,\bigr)^{\frac{1}{M'-1}}}.$}
\end{equation}

The estimates just obtained are uniform in \(\varepsilon\), since the constants depend only on the kernel quantities in Lemmas \ref{lemma:intDerivades}--\ref{lemma:canviVar}, on \(\|\gamma^\varepsilon\|_{H^3}\), and on \(\|F(\gamma^\varepsilon)\|_{L^\infty}\), but not on the mollification parameter. The same bounds hold for the regularized solutions \(\gamma^\varepsilon\). Therefore, one gets a time of existence which is independent of $\varepsilon$. The existence result now follows from the standard regularization and compactness argument for contour dynamics; see \cite{gancedo2008SQG, GNP22}. We briefly note that the passage to the limit applies to the present kernel class: after splitting the contour integral into $|x'|\geq\rho$ and $|x'|<\rho$, convergence on the first region follows from the smoothness of $K$ away from the origin, while the uniform arc-chord bound, the monotonicity of $|K|$ near the origin, and assumption \ref{ass:integrability} provide a uniform integrable majorant on the second region. The usual compactness and energy arguments then yield $\gamma\in C([0,T];H^3(\T))$.

    It remains to show that the solution is unique. Let $\gamma_1$ and $\gamma_2$ be two solutions of equation \eqref{eq:patch} with $\gamma_1(x,0)=\gamma_2(x,0),$ and $\gamma=\gamma_1-\gamma_2.$ Then 
    \begin{align*}
        \int_\T \gamma(x)\cdot\gamma_t(x)dx&= \int_\T\int_\T \gamma(x)\cdot (\partial_x\gamma_1(x)-\partial_x\gamma_1(x-x'))\\&\qquad\cdot(K(|\gamma_1(x)-\gamma_1(x-x')|)-K(|\gamma_2(x)-\gamma_2(x-x')|)) dx'dx\\ 
        &+ \int_\T\int_\T \gamma(x)\cdot (\partial_x\gamma(x)-\partial_x\gamma(x-x')) K(|\gamma_2(x)-\gamma_2(x-x')|)dx'dx\\
        &=I_7+I_8.
    \end{align*}
    Now for $I_7$ we do as for $I_5$ and use \eqref{eq:diferenciaKernelsRTI}:
    \begin{align*}
        |I_7|&\leq C\Ck{\gamma_1}\intTT |\gamma(x)||\gamma(x)-\gamma(x-x')| |x'|\\&\quad\cdot (|K'(|\gamma_1(x)-\gamma_1(x-x')|)|+|K'(|\gamma_2(x)-\gamma_2(x-x')|)|)dx'dx\\
        & \leq C \Ck{\gamma_1} (\Np{F(\gamma_1)}+\Np{F(\gamma_2)}) \\ &\quad\cdot \int_\T \frac{|x'|}{\Np{F(\gamma_1)}}\left|K'\left(\frac{|x'|}{\Np{F(\gamma_1)}}\right)\right|+\frac{|x'|}{\Np{F(\gamma_2)}}\left|K'\left(\frac{|x'|}{\Np{F(\gamma_2)}}\right)\right| \\&\qquad\cdot\int_\T |\gamma(x)||\gamma(x)-\gamma(x-x')|dxdx'\\
        &\leq C \Ck{\gamma_1} (1+\Np{F(\gamma_1)}^2+\Np{F(\gamma_2)}^2)\Lp{\gamma}^2,
    \end{align*}
    where in the last inequality we used Cauchy--Schwarz and Lemma \ref{lemma:canviVar}. Finally, for $I_8$ we proceed as for $I_1$: 
    \begin{align*}
        I_8&= \frac{1}{2} \int_\T\int_\T (\gamma(x)-\gamma(x-x'))\cdot(\partial_x\gamma(x)-\partial_x\gamma(x-x'))K(|\diff{\gamma_2}|) dx'dx\\
        &= \frac{1}{4} \int_\T\int_\T \partial_x[|\gamma(x)-\gamma(x-x')|^2]K(|\diff{\gamma_2}|) dx'dx\\
        &= -\frac{1}{4} \int_\T\int_\T |\gamma(x)-\gamma(x-x')|^2 K'(|\diff{\gamma_2}|) \frac{\diff{\gamma_2}\cdot\diff{\partial_x\gamma_2}}{|\diff{\gamma_2}|} dx'dx,
    \end{align*}
    so
    \begin{align*}
        |I_8|&\leq C \Ck{\gamma_2} \int_\T\int_\T |\gamma(x)-\gamma(x-x')|^2\left|K'\left(\frac{|x'|}{F(\gamma_2)}\right)\right||x'|dx'dx\\
        &\leq C \Ck{\gamma_2} \Np{F(\gamma_2)}\int_\T \frac{|x'|}{\Np{F(\gamma_2)}}\left|K'\left(\frac{|x'|}{\Np{F(\gamma_2)}}\right)\right|\int_\T |\gamma(x)-\gamma(x-x')|^2 dxdx' \\
        &\leq C \Ck{\gamma_2} (1+\Np{F(\gamma_2)}^2) \Lp{\gamma}^2 .
    \end{align*}

    Combining both estimates we obtain $$\frac{d}{dt}\Lp{\gamma}^2(t)\leq C \Lp{\gamma}^2(t),$$ and using Grönwall inequality we conclude that $\gamma=0.$

    \section{\texorpdfstring{Local existence in $H^2$}{Local existence in H2}} \label{sec:H2} 
 
    In this section we prove Theorem \ref{thm:H2}. That is, under the stronger integrability assumption \eqref{ass:strong-integrability} on the kernel $K$, we show that for an initial data $\gamma_0\in H^2(\T)$ with $F(\gamma_0)(x,x')$ bounded, there exists a local solution to the contour dynamics equation \eqref{eq:CDEmod}.

    First, notice that \eqref{ass:strong-integrability} implies \ref{ass:integrability} since by Hölder inequality 

    \begin{align*}
        \int_0^1 |K(r)|\,dr
=
\int_0^1 r^\alpha |K(r)|\,r^{-\alpha}\,dr
\le
\left(\int_0^1 r|K(r)|^{1/\alpha}\,dr\right)^\alpha
\left(\int_0^1 r^{-\frac{\alpha}{1-\alpha}}\,dr\right)^{1-\alpha},
    \end{align*}
    where the first term is bounded by the integrability condition and the second is finite because $\alpha\in (0,1/2)$ guarantees the convergence of the integral.

    To obtain the desired energy estimate in $H^2$, we choose the parameter $\lambda(x,t)$ from \eqref{eq:CDEmod} to get an extra cancellation. The presence of a non‑zero $\lambda$ adds several nontrivial terms to the estimates, but they can be managed with the help of the stronger integrability assumption \eqref{ass:strong-integrability} and the geometric bounds on the curve.
    
    The lower regularity in $H^2$ requires a finer control of the tangential motion, which is achieved by fixing the parametrization so that the tangent vector satisfies
    \begin{equation}\label{eq:extraCancellation}
        \partial_x \gamma (x,t)\cdot \partial_x^2 \gamma(x,t) =0.
    \end{equation}

    Given an initial curve satisfying \eqref{eq:arcChord}, we can reparametrize it obtaining that $|\partial_x\gamma(x,0)|^2=1,$ and therefore \eqref{eq:extraCancellation} is fulfilled at $t=0.$ In general, the condition $|\partial_x \gamma(x,t)|^2 = 1$ may not be preserved for all time; instead, we have
    \begin{equation} \label{eq:cancellationTemps}
        |\partial_x\gamma(x,t)|^2=A(t).
    \end{equation}

    Write $$NL(x,t)\coloneqq -\int_\T K(|\gamma(x)-\gamma(x-x')|)(\partial_x \gamma(x)-\partial_x \gamma(x-x'))dx'.$$

    Taking $\partial_x$ on \eqref{eq:CDEmod}, and not writing the time dependence in order to ease the notation, gives
    $$\partial_x \partial_t \gamma(x) = \partial_x NL(x) + \lambda (x) \partial_x^2\gamma(x) + \partial_x\lambda(x) \partial_x \gamma(x).$$

    By taking now inner product with $\partial_x\gamma(x)$  we have
    $$\partial_x\gamma(x)\cdot \partial_x \partial_t \gamma(x) = \partial_x\gamma(x)\cdot \partial_x NL(x) + \partial_x\gamma(x)\cdot\lambda (x) \partial_x^2\gamma(x) + \partial_x\gamma(x)\cdot\partial_x\lambda(x) \partial_x \gamma(x),$$
    where the second term vanishes since by \eqref{eq:cancellationTemps} $$\partial_x\gamma(x)\cdot\lambda (x) \partial_x^2\gamma(x)=\frac{1}{2}\lambda(x) \partial_x(|\partial_x \gamma(x)|^2)=\frac{1}{2}\lambda(x) \partial_x(A(t)).$$

    Therefore, 
     $$\partial_x\gamma(x)\cdot \partial_x \partial_t \gamma(x) = \partial_x\gamma(x)\cdot \partial_x NL(x) + \partial_x\lambda(x) |\partial_x \gamma(x)|^2,$$ and integrating on the torus we get for the left hand side
     \begin{multline*}
         \int_\T \partial_x\gamma(x)\cdot \partial_x \partial_t \gamma(x) dx = \frac{1}{2}\int_\T  \partial_t (|\partial_x \gamma(x)|^2) dx \\= \frac{1}{2}\int_\T  \partial_t (A(t)) dx = \pi\partial_t (A(t)) =2\pi \partial_x\gamma(x)\cdot \partial_x\partial_t\gamma(x), 
     \end{multline*}

    and due to the fact that $\lambda$ has to be periodic we have for the right hand side
    $$\int_\T \partial_x\gamma(x)\cdot \partial_x NL(x)dx + \int_\T \partial_x\lambda(x) A(t) dx = \int_\T \partial_x\gamma(x)\cdot \partial_x NL(x)dx.$$

    So combining both with \eqref{eq:CDEmod} yields $$2\pi \partial_x\gamma(x)\cdot \partial_x\partial_t\gamma(x)= 2\pi \partial_x\gamma(x)\cdot \partial_x NL(x)+2\pi \partial_x \lambda(x) A(t) = \int_{-\pi}^\pi \partial_y\gamma(y)\cdot \partial_y NL(y)dy,$$

    so 
    \begin{equation} \label{eq:lambda}
        \lambda(x)=\frac{x+\pi}{2\pi} \int_\T \frac{\partial_y \gamma(y) \cdot \partial_y NL(y)}{A(t)} dy - \int_{-\pi}^x \frac{\partial_y \gamma(y)\cdot \partial_y NL(y)}{A(t)}dy.
    \end{equation}

    With the above parametrization in place, we are now prepared to investigate the temporal evolution of the quantity $\Hd{\gamma}^2 = \Lp{\gamma}^2 + \Lp{\partial_x^2 \gamma}^2.$
    Consider first
    \begin{equation}
    \begin{aligned}
        \int_\T \gamma(x)\cdot\gamma_t(x)dx &= \int_\T \gamma(x) \cdot NL(x) dx + \intT \lambda(x)\gamma(x)\cdot\partial_x\gamma(x) dx,
    \end{aligned}
    \end{equation}

    and  proceeding as for the $H^3$ case (see \eqref{eq:gammaL2}) one gets that the first term vanishes. Therefore, by integration by parts 
    \begin{equation}\label{eq:H2L2}
        \partial_t \Lp{\gamma}^2(t)\leq \Np{\partial_x \lambda}(t)\Lp{\gamma}^2(t).
    \end{equation}

    We aim now to bound $\Np{\partial_x \lambda}(t)$. To that end, by product rule and the Fundamental Theorem of Calculus we have 
    \begin{align*}
        \partial_x \lambda &= \frac{1}{2\pi} \int_\T \frac{\partial_x \gamma(x)}{|\partial_x\gamma(x)|^2} \cdot \partial_x\left(\intT -K(|\diff{\gamma}|) \diff{\partial_x\gamma} dx'\right) dx\\&\quad+\frac{\partial_x\gamma(x)}{|\partial_x\gamma(x)|^2} \cdot\intT \diff{\partial_x^2\gamma} K(|\diff{\gamma}|)dx'\\&\quad-\frac{\partial_x\gamma(x)}{|\partial_x\gamma(x)|^2} \cdot\intT \diff{\partial_x\gamma }\frac{\diff{\gamma}\cdot\diff{\partial_x\gamma}}{|\diff{\gamma}|}K'(|\diff{\gamma}|)dx'\eqqcolon A_0+A_1+A_2.
    \end{align*}

    Notice that by \eqref{eq:extraCancellation} we have $$-\partial_x\gamma(x)\cdot(\partial_x^2\gamma(x)-\partial_x^2\gamma(x-x'))=(\partial_x\gamma(x)-\partial_x\gamma(x-x'))\cdot\partial_x^2\gamma(x-x'),$$ thus,

    $$A_1=\intT \frac{\diff{\partial_x \gamma}\cdot \partial_x^2\gamma (x-x')}{|\partial_x\gamma(x)|^2} K(|\diff{\gamma}|)dx'.$$

    So, for any $\alpha\in\left(0,\frac{1}{2}\right]$ we have
    \begin{align*}
        |A_{1}|&\leq\intT \frac{|\diff{\partial_x \gamma}||\partial_x^2\gamma (x-x')|}{|\partial_x\gamma(x)|^2} |K(|\diff{\gamma}|)|dx'\\
        & \leq \Np{F(\gamma)}^2 \intT \frac{|\diff{\partial_x \gamma}||\partial_x^2\gamma (x-x')||\diff{\gamma}|^2}{|\partial_x\gamma(x)|^2|x'|^2} |K(|\diff{\gamma}|)|dx' \\
        &\leq \Np{F(\gamma)}^2 \SNa[\alpha]{\partial_x\gamma}\intT |x'|^\alpha |\partial_x^2\gamma (x-x')||K(|\diff{\gamma}|)|dx',
    \end{align*}
     where in the third inequality we use once more \eqref{eq:cancellationTemps}, i.e. $|\partial_x\gamma(x,t)|^2=A(t)$, and Sobolev embeddings. To finish the bound we fix $\alpha$ so that the extra integrability \eqref{ass:strong-integrability} holds and then we split as in $I_6$ above (see \eqref{eq:splitKernel3regions}) and bound each region separately

     \begin{align*}
         &\int_{|\xi| \le \delta/\Np{\partial_x\gamma}} |x'|^\alpha |\partial_x^2\gamma (x-x')|K\left(|\diff{\gamma}|\right)dx'\\
         &\leq \Np{F(\gamma)}^\alpha \int_{|\xi| \le \delta/\Np{\partial_x\gamma}} \left(\frac{|x'|}{\Np{F(\gamma)}}\right)^\alpha |\partial_x^2\gamma (x-x')| \left|K\left(\frac{|x'|}{\Np{F(\gamma)}}\right)\right| dx'
         \\
         &\leq \Np{F(\gamma)}^\alpha \left(\int_{|\xi| \le \delta/\Np{\partial_x\gamma}} \left(\frac{|x'|}{\Np{F(\gamma)}}\right) \left|K\left(\frac{|x'|}{\Np{F(\gamma)}}\right)\right|^{\frac{1}{\alpha}} dx'\right)^{\alpha} \Lp[p]{\partial_x^2 \gamma}
         \\&\leq C \Np{F(\gamma)}^{1+\alpha} \Lp[p]{\partial_x^2 \gamma} \left(\int_0^\delta \nu |K(\nu)|^{\frac{1}{\alpha}} d\nu\right)^{\alpha}\leq C \Np{F(\gamma)}^{1+\alpha} \Lp[p]{\partial_x^2 \gamma},
     \end{align*}
     where in the second step we used Hölder inequality with $p=\frac{1}{1-\alpha}\in(1,2)$ and in the last step assumption \eqref{ass:strong-integrability}. For the other two regions we proceed analogously
     \begin{align*}
         &\int_{\delta/\|\partial_x\gamma\|_{L^\infty}}^{\delta\|F(\gamma)\|_{L^\infty}} |x'|^\alpha |\partial_x^2\gamma (x-x')| \left|K\left(|\diff{\gamma}|\right)\right| dx'\\
         &\leq \int_{\delta/\|\partial_x\gamma\|_{L^\infty}}^{\delta\|F(\gamma)\|_{L^\infty}} |x'|^\alpha |\partial_x^2\gamma (x-x')|\left(\left|K\left(\frac{|x'|}{\Np{F(\gamma)}}\right)\right|+C(1+\Np{\partial_x\gamma}^M|x'|^M)\right)dx',
     \end{align*}
    where the first part is bounded by the same argument and for the second part we also use Hölder to get  
    \begin{align*}
        &C \int_{\delta/\|\partial_x\gamma\|_{L^\infty}}^{\delta\|F(\gamma)\|_{L^\infty}} |x'|^\alpha |\partial_x^2\gamma (x-x')|(1+\Np{\partial_x\gamma}^M|x'|^M)dx' \\&\leq C \int_0^\pi |x'|^\alpha |\partial_x^2\gamma (x-x')|(1+\Np{\partial_x\gamma}^M|x'|^M)dx' \leq C\Lp[p]{\partial_x^2 \gamma}(1+\Np{\partial_x\gamma}^M).
    \end{align*}

    The last region is treated as the second term of the above estimate. Therefore, $$|A_1|\leq\Np{F(\gamma)}^2 \Na[\alpha]{\partial_x\gamma}\Lp[p]{\partial_x^2 \gamma} (\Np{F(\gamma)}^{1+\alpha}+1+\Np{\partial_x\gamma}^M).$$

    For $A_2,$ we proceed as follows
    \begin{align*}
        |A_2|&\leq \Na[\sigma]{\partial_x\gamma }^2 \intT \frac{|x'|^{2\sigma}|x'|}{|\diff{\gamma}|}| K'(|\diff{\gamma}|)|dx' \\ 
        &\leq \Na[\sigma]{\partial_x\gamma }^2 \Np{F(\gamma)}^{1+2\sigma} \left|\intT \left(\frac{|x'|}{\Np{F(\gamma)}}\right)^{2\sigma}K'\left(\frac{|x'|}{\Np{F(\gamma)}}\right) dx'\right|\\
        &\leq C\Na[\sigma]{\partial_x\gamma }^2 \Np{F(\gamma)}^{2+2\sigma} \int_0^{\frac{\pi}{\Np{F(\gamma)}}} \nu^{2\sigma} |K'(\nu)| d\nu ,
    \end{align*}

    taking $\sigma=\frac{1}{2}$ gives that by lemma \ref{lemma:canviVar} the last integral is bounded and 
    $$|A_2|\leq  C\Na[1/2]{\partial_x\gamma }^2 \Np{F(\gamma)}(1+\Np{F(\gamma)}^{2}).$$

    A routine computation shows that \(A_0\) behaves like the combination of the \(A_1\) and \(A_2\) terms and therefore 
    \begin{equation}
        \label{eq:controlDerivadaLambda}
        \begin{aligned}       
        \Np{\partial_x\lambda} &\leq C (\Np{F(\gamma)}^2 \Na[\alpha]{\partial_x\gamma}\Lp[p]{\partial_x^2 \gamma} (\Np{F(\gamma)}^{1+\alpha}+1+\Np{\partial_x\gamma}^M) \\&\quad+\Na[1/2]{\partial_x\gamma }^2 \Np{F(\gamma)}(1+\Np{F(\gamma)}^{2})).
         \end{aligned}
    \end{equation}

    The bound on $\Np{\partial_x \lambda}$ derives the bound for \eqref{eq:H2L2}. We control now the quantity $$-\int_\T \partial_x^2\gamma(x)\cdot\partial_x^2\gamma_t(x)dx= I_1+ I_2+ I_3+I_4,$$ where 
    \begin{align*}        I_1&=\int_\T\int_\T\partial_x^2\gamma(x)\cdot(\partial_x^3\gamma(x)-\partial_x^3\gamma(x-x'))K(|\gamma(x)-\gamma(x-x')|)dx'dx,\\         I_2&=2\int_\T\int_\T\partial_x^2\gamma(x)\cdot(\partial_x^2\gamma(x)-\partial_x^2\gamma(x-x'))\partial_x[K(|\gamma(x)-\gamma(x-x')|)]dx'dx,\\        I_3&=\int_\T\int_\T\partial_x^2\gamma(x)\cdot(\partial_x\gamma(x)-\partial_x\gamma(x-x'))\partial_x^2[K(|\gamma(x)-\gamma(x-x')|)]dx'dx,\\
        I_4&=\frac{3}{2}\intT |\partial_x^2\gamma(x)|^2\partial_x\lambda(x) dx .       
    \end{align*}

    In the above, $I_4$ follows from the extra cancellation \eqref{eq:extraCancellation} and integration by parts.
    
    To bound $I_1$ we start by symmetrizing and integrating by parts, which yields
    \begin{align*}
        I_1&=\frac{1}{2}\intTT \diff{\partial_x^2\gamma}\cdot\diff{\partial_x^3\gamma} K(|\diff{\gamma}|)dx'dx \\
         &=\frac{1}{4}\intTT \partial_x [|\diff{\partial_x^2\gamma}|^2] K(|\diff{\gamma}|)dx'dx \\
         &=-\frac{1}{4} \intTT |\diff{\partial_x^2\gamma}|^2 \frac{\diff{\gamma} \cdot \diff{\partial_x\gamma}}{|\diff{\gamma}|} K'(|\diff{\gamma}|) dx'dx.
    \end{align*}
    
    We now decompose 
    \begin{equation}    \label{eq:decompositionDiff}
    \diff{\gamma}\cdot \diff{\partial_x\gamma}=(\diff{\gamma}-\partial_x\gamma(x)x')\cdot \diff{\partial_x\gamma}+\partial_x\gamma(x)x'\cdot \diff{\partial_x\gamma},
    \end{equation}
    and notice that by the extra cancellation \eqref{eq:cancellationTemps} once more 
    \begin{equation}\label{eq:cotaProdUniDif}
        \partial_x\gamma(x)\cdot \diff{\partial_x\gamma} = \frac{1}{2}|\diff{\partial_x\gamma}|^2.
    \end{equation}

    Indeed, \begin{equation*}
        \begin{aligned}
                   |\diff{\partial_x\gamma}|^2&=|\partial_x \gamma(x)|^2+|\partial_x\gamma(x-x')|^2-2|\partial_x\gamma(x)\cdot\partial_x\gamma(x-x')\\
                   &=2|\partial_x\gamma(x)|^2-2|\partial_x\gamma(x)\cdot\partial_x\gamma(x-x')=2\partial_x\gamma(x)\cdot \diff{\partial_x\gamma}.
        \end{aligned}
    \end{equation*}

    Combining \eqref{eq:decompositionDiff} and \eqref{eq:cotaProdUniDif} with the fundamental theorem of calculus one gets the following bound for all $\sigma\in\left(0,\frac{1}{2}\right]$  
    \begin{equation}\label{eq:cotaProdDiferencies}
        |\diff{\gamma}\cdot\diff{\partial_x \gamma} | \leq \frac{3}{2}|x'|^{1+2\sigma} \SNa[\sigma]{\partial_ x \gamma }^2.
    \end{equation}
    
    Using this bound we get 
    \begin{align*}
        |I_1|&\leq C \intTT |\diff{\partial_x^2\gamma}|^2 \frac{|x'|^{1+2\sigma} \SNa[\sigma]{\partial_ x \gamma }^2}{|\diff{\gamma}|} |K'(|\diff{\gamma}|)| dx'dx.
        \\
        &= C \SNa[\sigma]{\partial_ x \gamma }^2 \intTT |\diff{\partial_x^2\gamma}|^2 |F(\gamma)|^{1+2\sigma}\left(\frac{|x'|}{|F(\gamma)|}\right)^{2\sigma} \left|K'\left(\frac{|x'|}{|F(\gamma)|}\right)\right| dx'dx \\
        &\leq  C \SNa[\sigma]{\partial_ x \gamma }^2 \Np{F(\gamma)}^{1+2\sigma} \int_\T\left(\frac{|x'|}{\Np{F(\gamma)}}\right)^{2\sigma} \left| K'\left(\frac{|x'|}{\Np{F(\gamma)}}\right) \right| \int_\T |\diff{\partial_x^2\gamma}|^2  dxdx'\\
        &\leq C \SNa[\sigma]{\partial_ x \gamma }^2 \Np{F(\gamma)}^{1+2\sigma}  \Lp{\partial_x^2 \gamma}^2 \int_\T\left(\frac{|x'|}{\Np{F(\gamma)}}\right)^{2\sigma} \left| K'\left(\frac{|x'|}{\Np{F(\gamma)}}\right) \right| dx',
    \end{align*}
     where we have used the monotonicity assumption of $K'$. Finally, taking for instance $\sigma=\frac{1}{2}$ we get an integral which we already bounded using Lemma \ref{lemma:canviVar} on $I_1$ from the $H^3$ case, so we have
     $$|I_1|\leq \SNa[1/2]{\partial_x\gamma}^2 \Lp{\partial_x^2\gamma}^2 (\Np{F(\gamma)}^3+\Np{F(\gamma)}). $$
         
    After symmetrizing on $I_2$ it is clear that $I_2=-4I_1$ so a similar bound holds. 
    
    We focus now on $I_3=J_1+J_2+J_3+J_4$ where
     \begin{align*}
            J_1&=\intTT\partial_x^2\gamma(x)\cdot \diff{\partial_x \gamma}K''(|\diff{\gamma}|) \frac{(\diff{\partial_x\gamma}\cdot\diff{\gamma})^2}{|\diff{\gamma}|^2}dx'dx,\\
            J_2&=\intTT \partial_ x^2\gamma(x) \cdot \diff{\partial_x\gamma}K'(|\diff{\gamma}|)\frac{\diff{\partial_x^2 \gamma}\cdot\diff{\gamma}}{|\diff{\gamma}|}dx'dx,\\
            J_3&=\intTT \partial_ x^2\gamma(x) \cdot \diff{\partial_x \gamma} K'(|\diff{\gamma}|)\frac{|\diff{\partial_x\gamma}|^2}{|\diff{\gamma}|}dx'dx,\\
            J_4&=-\intTT \partial_ x^2\gamma(x) \cdot \diff{\partial_x \gamma}K'(|\diff{\gamma}|)\frac{(\diff{\partial_x\gamma}\cdot\diff{\gamma})^2}{|\diff{\gamma}|^3}dx'dx.
        \end{align*}

        The following identity, analogous to \eqref{eq:MVTgammaH3} used in the previous section is key 
        \begin{equation}\label{eq:MVTgammaH2}
            \diff{\partial_x\gamma}=x' \int_0^1 \partial_x^2\gamma(x-sx') ds,
        \end{equation}

        with this and recalling that $|K''|$  is decreasing we have 
        \begin{align*}
            |J_1|&\leq \Na[\sigma]{\partial_x\gamma}^2\intT |x'|^{1+2\sigma} \left|K''\left(\frac{|x'|}{\Np{F(\gamma)}}\right)\right| \left(\intT\int_0^1 \partial_x^2\gamma(x) \cdot \partial_x^2\gamma(x-sx')ds dx\right) dx'\\
            &\leq \Na[\sigma]{\partial_x\gamma}^2 \Lp{\partial_x^2 \gamma}^2 \Np{F(\gamma)}^{1+2\sigma} \left|\intT \left(\frac{|x'|}{\Np{F(\gamma)}}\right)^{1+2\sigma} K''\left(\frac{|x'|}{\Np{F(\gamma)}}\right) dx'\right|,
        \end{align*}

       and, choosing again $\sigma=\frac{1}{2}$ and applying Lemma \ref{lemma:canviVar}, we obtain 
        $$|J_1|\leq \Na[1/2]{\partial_x\gamma}^2 \Lp{\partial_x^2 \gamma}^2 (\Np{F(\gamma)}^{3}+1).$$

        By the same reasoning, we derive analogous estimates for $J_3$ and $J_4$. Namely, $$|J_i|\leq \Na[1/2]{\partial_x\gamma}^2 \Lp{\partial_x^2 \gamma}^2 (\Np{F(\gamma)}^{3}+\Np{F(\gamma)}), \; i=3,4.$$
        
        For $J_2$, Notice that by \eqref{eq:extraCancellation} we have 
        $$\diff{\gamma}\cdot \diff{\partial_x^2 \gamma}=(\diff{\gamma}-\partial_x\gamma(x)x')\cdot \diff{\partial_x^2\gamma}-x'\diff{\partial_x\gamma}\cdot \partial_x^2\gamma (x-x')$$
        therefore
        $$\left|\diff{\gamma}\cdot\diff{\partial_x^2\gamma}\right|\leq C\SNa[\sigma]{\partial_x\gamma}|x'|^{1+\sigma}
\left(\left|\diff{\partial_x^2\gamma}\right|+\left|\partial_x^2\gamma(x-x')\right|\right).$$

        and

        \begin{align*}
            |J_2|&\leq \Np{F(\gamma)} \SNa[\sigma]{\partial_x\gamma }^2 \intTT |\partial_x^2\gamma(x)| \frac{|x'|^\sigma}{|x'|} |K'(|\diff{\gamma}|)|| \diff{\partial_x^2 \gamma}\cdot\diff{\gamma}| dx'dx\\
            &\leq \Np{F(\gamma)} \SNa[\sigma]{\partial_x\gamma }^2 \intTT |\partial_x^2\gamma(x)| |x'|^{2\sigma} |K'(|\diff{\gamma}|)| \left(\left|\diff{\partial_x^2\gamma}\right|+\left|\partial_x^2\gamma(x-x')\right|\right)dx'dx\\
            &\leq \Np{F(\gamma)} \SNa[\sigma]{\partial_x\gamma }^2 \intT  |x'|^{2\sigma} \left|K'\left(\frac{|x'|}{\Np{F(\gamma)}}\right)\right| dx'\\ &\qquad\qquad\cdot \intT |\partial_x^2\gamma(x)|\cdot\left(\left|\diff{\partial_x^2\gamma}\right|+\left|\partial_x^2\gamma(x-x')\right|\right)dx\\
            &\leq  \Np{F(\gamma)}^{1+2\sigma} \SNa[\sigma]{\partial_x\gamma }^2 \Lp{\partial^2_x \gamma}^2\intT  \left(\frac{|x'|}{\Np{F(\gamma)}}\right)^{2\sigma} \left|K'\left(\frac{|x'|}{\Np{F(\gamma)}}\right)\right| dx'.
        \end{align*}

        Finally, proceeding as before one gets $$|J_2|\leq \SNa[1/2]{\partial_x\gamma}^2 \Lp{\partial_x^2\gamma}^2 (\Np{F(\gamma)}^3+\Np{F(\gamma)}), $$

        and $$|I_3|\leq C \SNa[1/2]{\partial_x\gamma}^2 \Lp{\partial_x^2\gamma}^2 (\Np{F(\gamma)}^3+\Np{F(\gamma)}+1).$$      
                        
        For $I_4$ we use Cauchy--Schwarz to get $|I_4|\leq C \Np{\partial_x \lambda} \Lp{\partial_x^2 \gamma}^2,$ finally by means of \eqref{eq:controlDerivadaLambda} one gets a desired bound.
        
        Gathering all the $I_j$ estimates and using Sobolev embeddings we obtain 
        $$\frac{d}{dt} \Lp{\partial_x^2 \gamma}^2 \leq \mathcal{P}(\Hd{\gamma},\Np{F(\gamma)})\left(1+\Lp{\partial_x^2 \gamma}^2\right),$$ 
        where $\mathcal{P}$ is a polynomial with positive coefficients.

        It remains to deal with the arc-chord condition in this case. To do so, consider its evolution 
        $$\partial_tF(\gamma)=-\frac{|x'|(\gamma(x)-\gamma(x-x'))}{|\gamma(x)-\gamma(x-x')|^3}\cdot(\partial_t\gamma(x)-\partial_t\gamma(x-x'))=B_1+B_2+B_3+B_4+B_5,$$

         where the terms \(B_i\) come from inserting the right‑hand side of \eqref{eq:CDEmod} for \(\partial_t\gamma\):
        \begin{align*}
            B_1&= -\frac{|x'|\diff{\gamma}}{|\diff{\gamma}|^3}\cdot (\partial_x\gamma(x)-\partial_x\gamma(x-x')) \intT K(|\gamma(x)-\gamma(x-\xi)|)d\xi,\\
            B_2&=-\frac{|x'|\diff{\gamma}}{|\diff{\gamma}|^3}\cdot\intT (\partial_x\gamma(x-\xi)-\partial_x\gamma (x-x'-\xi)) K(|\gamma(x)-\gamma(x-\xi)|)d\xi,\\
            B_3&=-\frac{|x'|\diff{\gamma}}{|\diff{\gamma}|^3}\cdot \intT \diff[x'][x-\xi]{\partial_x\gamma} (K(|\diff[\xi][x]{\gamma}|)-K(|\diff[\xi][x-x']{\gamma}|)) d\xi,\\
            B_4&=-\frac{|x'|\diff{\gamma}}{|\diff{\gamma}|^3}\cdot (\lambda(x)-\lambda(x-x'))\partial_x\gamma(x),\\
            B_5&=-\frac{|x'|\diff{\gamma}}{|\diff{\gamma}|^3}\cdot\lambda(x-x')(\partial_x\gamma(x)-\partial_x\gamma(x-x')).
        \end{align*}

        We now bound each term separately. To that end, we proceed as in the previous estimates. By \eqref{eq:cotaProdDiferencies} and \eqref{eq:controlIntKernelDif} one gets
        \begin{align*}
        |B_1|&\leq C \frac{|x'|^3}{|\gamma(x)-\gamma(x-x')|^3}\SNa[1/2]{\partial_x\gamma}^2  \intT |K(|\gamma(x)-\gamma(x-\xi)|)|d\xi\\ 
        &\leq C F(\gamma) \Np{F(\gamma)}^2 \SNa[1/2]{\partial_x\gamma}^2 \intT |K(|\gamma(x)-\gamma(x-\xi)|)|d\xi\\
        &\leq C F(\gamma) \Np{F(\gamma)}^2 \SNa[1/2]{\partial_x\gamma}^2\left( \Np{F(\gamma)} + 1 + \Np{\partial_x\gamma}^M \right).
        \end{align*}

        We further split $B_2=B_{21}+B_{22}+B_{23}$ where
        \begin{align*}
            B_{21}&=\frac{|x'|(\diff{\gamma}-\partial_x\gamma(x)x')}{|\diff{\gamma}|^3}\cdot\intT (\partial_x\gamma(x-\xi)-\partial_x\gamma(x-x'-\xi))K(|\diff[\xi]{\gamma}|) d\xi, \\
            B_{22}&=\frac{|x'|x'}{|\diff{\gamma}|^3}\intT \diff[\xi]{\partial_x\gamma}\cdot(\partial_x\gamma(x-\xi)-\partial_x\gamma(x-x'-\xi))K(|\diff[\xi]{\gamma}|) d\xi, \\
            B_{23}&=\frac{|x'|x'}{|\diff{\gamma}|^3}\intT \partial_x\gamma(x-\xi)\cdot(\partial_x\gamma(x-\xi)-\partial_x\gamma(x-x'-\xi))K(|\diff[\xi]{\gamma}|) d\xi.
        \end{align*}

        For $B_{21}$ we proceed as before to get
        \begin{align*}
            |B_{21}|&\leq C F(\gamma) \Np{F(\gamma)}^2 \SNa[1/2]{\partial_x \gamma}^2 \intT |K(|\gamma(x)-\gamma(x-\xi)|)| d\xi\\ &\leq C F(\gamma) \Np{F(\gamma)}^2 \SNa[1/2]{\partial_x \gamma}^2 \left( \Np{F(\gamma)} + 1 + \Np{\partial_x\gamma}^M \right).
        \end{align*}

        For $B_{22}$ we first use \eqref{eq:MVTgammaH2} to get 
        \begin{align*}
            |B_{22}|&\leq C F(\gamma) \Np{F(\gamma)}^2 \SNa[\alpha]{\partial_x \gamma} \int_0^1\intT |\xi|^{\alpha}|\partial_x^2\gamma(x-\xi-sx')| |K(|\diff[\xi]{\gamma}|)|d\xi, 
        \end{align*}
        where $\alpha$ is the one from the extra integrability  \eqref{ass:strong-integrability}. Observe now that the remaining integral to be estimated is the same as the one we found for $A_1$, so carrying out the argument in the same way yields

        $$|B_{22}|\leq C F(\gamma) \Np{F(\gamma)}^2 \SNa[\alpha]{\partial_x \gamma} \Lp[p]{\partial_x^2 \gamma} (\Np{F(\gamma)}^{1+\alpha}+1+\Np{\partial_x\gamma}^M),$$

        Finally, using \eqref{eq:cotaProdUniDif} one gets 
        $$ |B_{23}|\leq C F(\gamma) \Np{F(\gamma)}^2 \SNa[1/2]{\partial_x \gamma}^2 \left( \Np{F(\gamma)} + 1 + \Np{\partial_x\gamma}^M \right).$$

        To bound $B_3$ we do as for $I_5$ from section \ref{sec:H3}. That is, we use \eqref{eq:diferenciaKernelsRTI} and the reverse triangular inequality to get 
        \begin{align*}
            |B_3|&\leq C F(\gamma) \Np{F(\gamma)} \SNa[1/2]{\partial_x\gamma}^2 \\&\quad  \cdot\intT |\xi|(|K'(|\gamma(x)-\gamma(x-\xi)|)|+|K'(|\gamma(x-x')-\gamma(x-x'-\xi)|)|) d\xi\\
            &\leq  C F(\gamma) \Np{F(\gamma)} \SNa[1/2]{\partial_x\gamma}^2 (\Np{F(\gamma)}^2+1).
        \end{align*}

        In $B_4$ the mean value inequality yields 
        $$|B_4|\leq F(\gamma) \Np{F(\gamma)} \Np{\partial_x\lambda} \Np{\partial_x\gamma},$$
        and by means of \eqref{eq:controlDerivadaLambda} one gets a bound of the form $F(\gamma)$ times a polynomial. 

        Using \eqref{eq:cotaProdDiferencies} in $B_5$ with $\sigma=1/2$ one gets $$|B_5|\leq C F(\gamma) \Np{F(\gamma)}^2 \Np{\lambda} \SNa[1/2]{\partial_x\gamma}^2,$$ so it remains to bound $\Np{\lambda}.$ To that end see that from \eqref{eq:lambda} we have
        \begin{align*}
            |\lambda(\gamma)|&\leq 2 \intT \left|\frac{\partial_x\gamma(x)}{|\partial_x\gamma(x)|^2} \cdot \intT \diff{\partial_x^2\gamma} K(|\diff{\gamma}|) dx'\right| dx \\&+ 2 \intT \left|\frac{\partial_x\gamma(x)}{|\partial_x\gamma(x)|^2} \cdot \intT \frac{\diff{\partial_x\gamma} \diff{\gamma} \cdot \diff{\partial_ x\gamma}}{|\diff{\gamma}|}K'(|\diff{\gamma}|) dx'\right| dx,
        \end{align*}
        where the first term is as $A_1$ and the second one as $A_2,$ therefore a similar bound to \eqref{eq:controlDerivadaLambda} holds for $\Np{\lambda}.$

        Gathering the previous estimates and by means of Sobolev embeddings (and recalling that $p<2$) we obtain 
        $$\frac{d}{dt}(\Hd{\gamma}+\Np{F(\gamma)})\leq \mathcal{P}_2(\Hd{\gamma}+\Np{F(\gamma)}),$$ 
        with $\mathcal{P}_2$ a polynomial with positive coefficients. Thus, it is possible to integrate the estimate to get a uniform bound for $\Hd{\gamma}+\Np{F(\gamma)}$ for a time $T>0$ sufficiently small depending only on $\Hd{\gamma_0}+\Np{F(\gamma_0)}$, as done for the $H^3$ case in \eqref{eq:integralLocalExistenceH3}. By an approximation argument, as in section \ref{sec:H3}, those a priori energy estimates provide the existence result, which concludes the proof of Theorem \ref{thm:H2}.

        \appendix

\section{Further examples of admissible kernels}\label{app:examples}
In this appendix we discuss some additional examples of radial kernels satisfying the assumptions used in the paper. We first recall a useful criterion based on complete monotonicity, which makes the monotonicity assumptions on the derivatives essentially automatic. This criterion is inspired by the approach of \cite{hmidi2023unifiedtheoryvstatesstructures}, where complete monotonicity is assumed. Here we observe that this hypothesis automatically implies the monotonicity assumptions in our framework, allowing it to serve as a convenient tool for verifying our hypotheses for a broader class of kernels. We then apply this criterion to a screened version of the gSQG kernel.

\subsection{Completely monotone kernels}

We recall that a function \(f:(0,\infty)\to\mathbb{R}\) is said to be completely monotone if
\(f\in C^\infty((0,\infty))\) and
\[
    (-1)^m f^{(m)}(r)\ge 0,
    \qquad r>0,\quad m=0,1,2,\ldots.
\]
By Bernstein's theorem, this is equivalent to the existence of a unique nonnegative Borel measure
\(\mu\) on \([0,\infty)\) such that
\[
    f(r)=\int_{[0,\infty)} e^{-rs}\,d\mu(s),
    \qquad r>0.
\]
We refer to \cite[Chapter 1]{schilling2010bernstein} for a systematic account of completely monotone and
Bernstein functions.

Complete monotonicity provides a convenient criterion for verifying our assumptions. If
\(-K'\) is completely monotone, then
\[
    (-1)^j K^{(j)}(r)\ge 0,
    \qquad r>0,\quad j\ge 1.
\]
Consequently,
\[
    |K^{(j)}(r)|=(-1)^jK^{(j)}(r),
    \qquad
    \frac{d}{dr}|K^{(j)}(r)|
    =
    (-1)^jK^{(j+1)}(r)\le0,
\]
so \(r\mapsto |K^{(j)}(r)|\) is nonincreasing for every \(j\ge1\). Thus the monotonicity assumption
\ref{ass:monotonicity} is automatically satisfied. If, in addition, \(K\) itself is completely monotone, then \(K\) is nonnegative and nonincreasing, and the growth condition in \ref{ass:compKernel} is automatic: for any fixed \(\delta>0\), one has \(K(r)\le K(\delta)\) for \(r\ge\delta\). Hence, for completely monotone kernels, the only remaining condition to check is the local integrability near the origin.

The classical kernels discussed in the introduction fit naturally into this framework. For the Euler kernel,
\(K(r)=-(2\pi)^{-1}\log r\), the function \(-K'(r)=(2\pi r)^{-1}\) is completely monotone. For the gSQG kernel \(K(r)=c_\beta r^{-\beta}\), \(0<\beta<1\), the kernel itself is completely monotone and is locally integrable at the origin. For the QGSW kernel, complete monotonicity follows from the standard integral representation of the modified Bessel function \(\mathbf{K}_0\). See \cite{hmidi2023unifiedtheoryvstatesstructures} for a detailed treatment of the classical models.

We now list some further admissible examples. In the first three cases, the stronger integrability condition of Theorem~\ref{thm:H2} is also satisfied, and therefore both Theorem~\ref{thm:H3} and Theorem~\ref{thm:H2} apply. The last example satisfies the assumptions of Theorem~\ref{thm:H3}, but not the stronger condition required in Theorem~\ref{thm:H2}.

\begin{itemize}
    \item Let \(0<a<1\) and \(b>0\). The kernels
    \[
        K(r)=r^{-a}e^{-r},
        \qquad
        K(r)=r^{-a}(1+r)^{-b}
    \]
    are completely monotone. Indeed, \(r^{-a}\), \(e^{-r}\), and \((1+r)^{-b}\) are completely monotone, and products of completely monotone functions are completely monotone. In both cases,
    \[
        K(r)\sim r^{-a}
        \qquad\text{as }r\to0^+.
    \]
    Hence \(K\in L^1(0,1)\) precisely when \(a<1\). Moreover, the stronger condition in Theorem~\ref{thm:H2} also holds, since one can choose
    \[
        \alpha\in\left(\frac{a}{2},\frac12\right).
    \]

    \item Let \(0\le\nu<1\). The kernel
    \[
        K(r)=\mathbf{K}_{\nu}(r),
    \]
    where \(\mathbf{K}_{\nu}\) denotes the modified Bessel function of the second kind, is completely monotone on \((0,\infty)\). Moreover,
    \[
        \mathbf{K}_{0}(r)\sim -\log r,
        \qquad
        \mathbf{K}_{\nu}(r)\sim C_\nu r^{-\nu}
        \quad\text{if }0<\nu<1,
        \qquad r\to0^+.
    \]
    Hence \(\mathbf{K}_{\nu}\in L^1(0,1)\) for \(0\le\nu<1\). The stronger condition in Theorem~\ref{thm:H2} is also satisfied: for \(0<\nu<1\), choose
    \[
        \alpha\in\left(\frac{\nu}{2},\frac12\right),
    \]
    while the case \(\nu=0\) follows from the logarithmic singularity.

    \item Let
    \[
        0\le b\le1,
        \qquad
        b\le a<b+1.
    \]
    Then
    \[
        K(r)=r^{-a}\bigl(\log(1+r)\bigr)^b
    \]
    is completely monotone. To see this, write
    \[
        K(r)
        =
        r^{-(a-b)}
        \left(\frac{\log(1+r)}{r}\right)^b.
    \]
    The factor \(r^{-(a-b)}\) is completely monotone because \(a-b\ge0\). On the other hand,
    \[
        \frac{\log(1+r)}{r}
        =
        \int_0^1 \frac{1}{1+rs}\,ds,
    \]
    and hence it is completely monotone. Since \(0\le b\le1\), its power
    \[
        \left(\frac{\log(1+r)}{r}\right)^b
    \]
    is also completely monotone. Therefore \(K\) is completely monotone as a product of completely monotone functions.

    Finally, since
    \[
        \log(1+r)\sim r
        \qquad\text{as }r\to0^+,
    \]
    one has
    \[
        K(r)\sim r^{b-a}
        \qquad\text{as }r\to0^+.
    \]
    Thus
    \[
        \int_0^1 K(r)\,dr<\infty
    \]
    if and only if \(b-a>-1\), that is, \(a<b+1\), which is exactly the upper restriction above. The stronger condition in Theorem~\ref{thm:H2} also holds. Indeed, if \(a\le b\), the kernel is bounded near the origin; if \(a>b\), one can choose
    \[
        \alpha\in\left(\frac{a-b}{2},\frac12\right),
    \]
    which is possible because \(a-b<1\).

    \item Let \(p>1\) and define
    \[
        K_p(r)=\int_e^\infty e^{-rt}\frac{dt}{(\log t)^p}.
    \]
    Then \(K_p\) is completely monotone by Bernstein's theorem. Moreover,
    \[
        K_p(r)\sim \frac{1}{r(\log(1/r))^p},
        \qquad r\to0^+.
    \]
    Hence \(K_p\in L^1(0,1)\) if and only if \(p>1\), so the basic integrability assumption holds. On the other hand, for every \(\alpha\in(0,1/2)\),
    \[
        rK_p(r)^{1/\alpha}
        \sim
        r^{1-\frac1\alpha}
        (\log(1/r))^{-\frac{p}{\alpha}},
        \qquad r\to0^+.
    \]
    Since \(\alpha<1/2\), one has \(1-\frac1\alpha<-1\), and the last expression is not integrable at zero. Therefore \(K_p\) satisfies the assumptions of Theorem~\ref{thm:H3}, but not the stronger integrability condition required in Theorem~\ref{thm:H2}.
\end{itemize}

\subsection{A screened gSQG kernel}

Let \(0<\beta<1\) and \(\varepsilon>0\). Consider the active scalar equation
\[
    \partial_t\omega+v\cdot\nabla\omega=0,
    \qquad
    v=\nabla^\perp(\varepsilon^2-\Delta)^{-1+\frac{\beta}{2}}\omega .
\]
Equivalently, the stream function is given by
\[
    \psi=(\varepsilon^2-\Delta)^{-1+\frac{\beta}{2}}\omega .
\]
This may be viewed as a screened version of the gSQG equation, in the same spirit in which the QGSW equation is obtained from the Euler kernel by replacing \((-\Delta)^{-1}\) with \((\varepsilon^2-\Delta)^{-1}\). Here, the replacement of \(-\Delta\) by \(\varepsilon^2-\Delta\) preserves the gSQG-type singularity at small scales, while introducing exponential decay at large distances.

We now identify the corresponding radial kernel. We use the Fourier convention
\[
    \widehat f(\xi)=\int_{\mathbb{R}^2}e^{-ix\cdot\xi}f(x)\,dx,
    \qquad
    f(x)=\frac{1}{(2\pi)^2}\int_{\mathbb{R}^2}e^{ix\cdot\xi}\widehat f(\xi)\,d\xi .
\]
The operator
\[
    (\varepsilon^2-\Delta)^{-1+\frac{\beta}{2}}
\]
has multiplier
\[
    (\varepsilon^2+|\xi|^2)^{-1+\frac{\beta}{2}}.
\]
Setting
\[
    \nu=1-\frac{\beta}{2},
\]
the associated kernel is the inverse Fourier transform of
\[
    (\varepsilon^2+|\xi|^2)^{-\nu}.
\]
Since the multiplier is radial, the kernel is radial as well. Using polar coordinates in Fourier space and the standard Hankel transform formula for Bessel potentials, one obtains
\[
    G_{\varepsilon,\beta}(r)
    =
    c_\beta\,
    \varepsilon^{\frac{\beta}{2}}
    r^{-\frac{\beta}{2}}
    \mathbf{K}_{\frac{\beta}{2}}(\varepsilon r),
    \qquad r>0,
\]
where \(\mathbf{K}_{\alpha}\) denotes the modified Bessel function of the second kind and, with the above convention,
\[
    c_\beta
    =
    \frac{2^{\frac{\beta}{2}}}
    {2\pi\,\Gamma(1-\frac{\beta}{2})}.
\]
Thus the stream kernel has the form
\[
    K(x,y)=G_{\varepsilon,\beta}(|x-y|).
\]
We refer to \cite[Chapter 9]{abramowitz1988handbook} for the classical identities for modified Bessel functions; see also \cite{watson1922treatise}.

The standard asymptotics of \(\mathbf{K}_{\alpha}\) imply
\[
    G_{\varepsilon,\beta}(r)\sim C_\beta r^{-\beta},
    \qquad r\to0^+,
\]
and
\[
    G_{\varepsilon,\beta}(r)
    =
    O\left(
        r^{-\frac{\beta+1}{2}}e^{-\varepsilon r}
    \right),
    \qquad r\to+\infty.
\]
Hence the screened kernel has the same local singularity as the usual gSQG kernel with parameter \(\beta\), but it is exponentially screened at infinity.

It follows immediately that
\[
    \int_0^1 G_{\varepsilon,\beta}(r)\,dr<\infty
\]
precisely in the range \(0<\beta<1\). Moreover, the stronger integrability condition in Theorem~\ref{thm:H2} is also satisfied: near the origin,
\[
    r\,G_{\varepsilon,\beta}(r)^{1/\alpha}
    \sim
    C r^{1-\frac{\beta}{\alpha}},
\]
which is integrable at zero whenever
\[
    \alpha>\frac{\beta}{2}.
\]
Since \(0<\beta<1\), one can choose
\[
    \alpha\in\left(\frac{\beta}{2},\frac12\right).
\]

Finally, this kernel is completely monotone since it is the product of completely monotone functions. Therefore \(G_{\varepsilon,\beta}\) satisfies the monotonicity and growth assumptions by the criterion of the previous subsection. Consequently, the screened gSQG kernels are covered by both Theorem~\ref{thm:H3} and Theorem~\ref{thm:H2}.

        \section*{Funding and Acknowledgments}
        The author is deeply grateful to Professor Taoufik Hmidi for his hospitality during a research visit to NYU Abu Dhabi. The problem studied in this paper originated during that visit, and many valuable discussions with him shaped the formulation of the problem, the assumptions on the kernel, and the verification of several technical estimates.
        
        The author also wishes to thank Professors Joan Mateu and Joan Orobitg for their invaluable support, encouragement, and insightful guidance throughout this research.
        
        This research is supported by the projects ``Análisis y Ecuaciones en Derivadas Parciales'' (Refs. PID2020-112881GB-I00 and PID2024-155320NB-I00) of the Spanish Ministry of Science, Innovation and Universities.
        
        \bibliographystyle{plain} 
        \bibliography{bibTesi}  

@book{MajdaBertozzi,
    AUTHOR = {Majda, Andrew J. and Bertozzi, Andrea L.},
    TITLE = {Vorticity and incompressible flow},
    SERIES = {Cambridge Texts in Applied Mathematics},
    VOLUME = {27},
    PUBLISHER = {Cambridge University Press, Cambridge},
    YEAR = {2002},
    PAGES = {xii+545},
}

@article{cantero2021regularity,
    AUTHOR = {Cantero, Juan Carlos and Mateu, Joan and Orobitg, Joan and
              Verdera, Joan},
    TITLE = {The regularity of the boundary of vortex patches for some
              nonlinear transport equations},
   JOURNAL = {Analysis \& PDE},
   VOLUME = {16},
   YEAR = {2023},
   NUMBER = {7},
   PAGES = {1621--1650},
   DOI = {10.2140/apde.2023.16.1621},
}

@article{bertozzi2016regularity,
  title={The regularity of the boundary of a multidimensional aggregation patch},
  author={Bertozzi, Andrea L. and Garnett, John B. and Laurent, Thomas and Verdera, Joan},
  journal={SIAM Journal on Mathematical Analysis},
  volume={48},
  number={6},
  pages={3789--3819},
  year={2016},
  publisher={SIAM},
  DOI = {10.1137/15M1033125},
}

@article{bertozzi2012aggregation,
   AUTHOR = {Bertozzi, Andrea L. and Laurent, Thomas and L\'eger, Flavien},
     TITLE = {Aggregation and spreading via the {N}ewtonian potential: the
              dynamics of patch solutions},
   JOURNAL = {Math. Models Methods Appl. Sci.},
  FJOURNAL = {Mathematical Models and Methods in Applied Sciences},
    VOLUME = {22},
      YEAR = {2012},
     PAGES = {1140005, 39},
       DOI = {10.1142/S0218202511400057},
}

@book{abramowitz1988handbook,
  TITLE = {Handbook of mathematical functions with formulas, graphs, and
              mathematical tables},
    EDITOR = {Abramowitz, Milton and Stegun, Irene A.},
      NOTE = {Reprint of the 1972 edition},
 PUBLISHER = {Dover Publications, Inc., New York},
      YEAR = {1992},
     PAGES = {xiv+1046},
      ISBN = {0-486-61272-4},
   MRCLASS = {00A20 (00A22 33-00)},
  MRNUMBER = {1225604},
}

@article{chemin1993persistance,
    AUTHOR = {Chemin, Jean-Yves},
     TITLE = {Persistance de structures g\'eom\'etriques dans les fluides
              incompressibles bidimensionnels},
   JOURNAL = {Ann. Sci. \'Ecole Norm. Sup. (4)},
  FJOURNAL = {Annales Scientifiques de l'\'Ecole Normale Sup\'erieure.
              Quatri\`eme S\'erie},
    VOLUME = {26},
      YEAR = {1993},
    NUMBER = {4},
     PAGES = {517--542},
      ISSN = {0012-9593},
   MRCLASS = {35Q35 (76C05)},
  MRNUMBER = {1235440},
MRREVIEWER = {Paolo\ Secchi},
       URL = {http://www.numdam.org/item?id=ASENS_1993_4_26_4_517_0},
}

@article{bertozzi1993global,
  AUTHOR = {Bertozzi, Andrea L. and Constantin, Peter},
     TITLE = {Global regularity for vortex patches},
   JOURNAL = {Comm. Math. Phys.},
  FJOURNAL = {Communications in Mathematical Physics},
    VOLUME = {152},
      YEAR = {1993},
    NUMBER = {1},
     PAGES = {19--28},
      ISSN = {0010-3616,1432-0916},
   MRCLASS = {35Q35 (76C99)},
  MRNUMBER = {1207667},
MRREVIEWER = {A.\ Elcrat},
       URL = {http://projecteuclid.org/euclid.cmp/1104252307},
}

@book{vallis2017atmospheric,
  title={Atmospheric and oceanic fluid dynamics: fundamentals and Large-Scale Circulation},
  author={Vallis, Geoffrey K.},
  year={2017},
  publisher={Cambridge University Press},
  ISBN={9781107588417},
  DOI={10.1017/9781107588417},
  PAGES = {xviii+946},
}

@book{watson1922treatise,
   AUTHOR = {Watson, George N.},
     TITLE = {A {T}reatise on the {T}heory of {B}essel {F}unctions},
 PUBLISHER = {Cambridge University Press, Cambridge; The Macmillan Company,
              New York},
      YEAR = {1944},
     PAGES = {vi+804},
   MRCLASS = {33.0X},
  MRNUMBER = {10746},
MRREVIEWER = {G.\ Szeg\"o},
}

@article {sQG95,
    AUTHOR = {Held, Isaac M. and Pierrehumbert, Raymond T. and Garner,
              Stephen T. and Swanson, Kyle L.},
     TITLE = {Surface quasi-geostrophic dynamics},
   JOURNAL = {J. Fluid Mech.},
  FJOURNAL = {Journal of Fluid Mechanics},
    VOLUME = {282},
      YEAR = {1995},
     PAGES = {1--20},
      ISSN = {0022-1120,1469-7645},
   MRCLASS = {76C05 (76F99 76U05 86A10)},
  MRNUMBER = {1312238},
       DOI = {10.1017/S0022112095000012},
       URL = {https://doi.org/10.1017/S0022112095000012},
}

@article {SQG2002,
    AUTHOR = {Lapeyre, Guillaume and Klein, Patrice},
     TITLE = {Dynamics of the upper oceanic layers in terms of surface
              quasigeostrophy theory},
   JOURNAL = {J. Phys. Oceanogr.},
  FJOURNAL = {Journal of Physical Oceanography},
    VOLUME = {36},
      YEAR = {2006},
    NUMBER = {2},
     PAGES = {165--176},
      ISSN = {0022-3670,1520-0485},
   MRCLASS = {86A10 (76B99)},
  MRNUMBER = {2214777},
       DOI = {10.1175/JPO2840.1},
       URL = {https://doi.org/10.1175/JPO2840.1},
}

@article {CordobaRodrigoaSQG,
    AUTHOR = {C\'ordoba, Diego and Fontelos, Marco A. and Mancho, Ana M. and
              Rodrigo, Jose L.},
     TITLE = {Evidence of singularities for a family of contour dynamics
              equations},
   JOURNAL = {Proc. Natl. Acad. Sci. USA},
  FJOURNAL = {Proceedings of the National Academy of Sciences of the United
              States of America},
    VOLUME = {102},
      YEAR = {2005},
    NUMBER = {17},
     PAGES = {5949--5952},
      ISSN = {0027-8424,1091-6490},
   MRCLASS = {76B03 (35A20 35Q35)},
  MRNUMBER = {2141918},
       DOI = {10.1073/pnas.0501977102},
       URL = {https://doi.org/10.1073/pnas.0501977102},
}

@article {depauw1999poche,
    AUTHOR = {Depauw, Nicolas},
     TITLE = {Poche de tourbillon pour {E}uler 2{D} dans un ouvert \`a{}
              bord},
   JOURNAL = {J. Math. Pures Appl. (9)},
  FJOURNAL = {Journal de Math\'ematiques Pures et Appliqu\'ees. Neuvi\`eme
              S\'erie},
    VOLUME = {78},
      YEAR = {1999},
    NUMBER = {3},
     PAGES = {313--351},
      ISSN = {0021-7824},
   MRCLASS = {76B47 (35Q30 76B03)},
  MRNUMBER = {1687165},
MRREVIEWER = {Marcel\ Oliver},
       DOI = {10.1016/S0021-7824(98)00003-8},
       URL = {https://doi.org/10.1016/S0021-7824(98)00003-8},
}

@article {kiselev2019global,
    AUTHOR = {Kiselev, Alexander and Li, Chao},
     TITLE = {Global regularity and fast small-scale formation for {E}uler
              patch equation in a smooth domain},
   JOURNAL = {Comm. Partial Differential Equations},
  FJOURNAL = {Communications in Partial Differential Equations},
    VOLUME = {44},
      YEAR = {2019},
    NUMBER = {4},
     PAGES = {279--308},
      ISSN = {0360-5302,1532-4133},
   MRCLASS = {35Q31 (35B65)},
  MRNUMBER = {3941226},
MRREVIEWER = {Francesco\ Fanelli},
       DOI = {10.1080/03605302.2018.1546318},
       URL = {https://doi.org/10.1080/03605302.2018.1546318},
}

@article {gancedo2008SQG,
    AUTHOR = {Gancedo, Francisco},
     TITLE = {Existence for the {$\alpha$}-patch model and the {QG} sharp
              front in {S}obolev spaces},
   JOURNAL = {Adv. Math.},
  FJOURNAL = {Advances in Mathematics},
    VOLUME = {217},
      YEAR = {2008},
    NUMBER = {6},
     PAGES = {2569--2598},
      ISSN = {0001-8708,1090-2082},
   MRCLASS = {35Q35 (76B03 76U05)},
  MRNUMBER = {2397460},
MRREVIEWER = {Milton\ C.\ Lopes Filho},
       DOI = {10.1016/j.aim.2007.10.010},
       URL = {https://doi.org/10.1016/j.aim.2007.10.010},
}

@article {chae2012generalized,
    AUTHOR = {Chae, Dongho and Constantin, Peter and C\'ordoba, Diego and
              Gancedo, Francisco and Wu, Jiahong},
     TITLE = {Generalized surface quasi-geostrophic equations with singular
              velocities},
   JOURNAL = {Comm. Pure Appl. Math.},
  FJOURNAL = {Communications on Pure and Applied Mathematics},
    VOLUME = {65},
      YEAR = {2012},
    NUMBER = {8},
     PAGES = {1037--1066},
      ISSN = {0010-3640,1097-0312},
   MRCLASS = {35Q86 (35A01 35A02 35A09 35D30 86A04)},
  MRNUMBER = {2928091},
MRREVIEWER = {Animikh\ Biswas},
       DOI = {10.1002/cpa.21390},
       URL = {https://doi.org/10.1002/cpa.21390},
}

@article {gancedo2021local,
    AUTHOR = {Gancedo, Francisco and Patel, Neel},
     TITLE = {On the local existence and blow-up for generalized {SQG}
              patches},
   JOURNAL = {Ann. PDE},
  FJOURNAL = {Annals of PDE. Journal Dedicated to the Analysis of Problems
              from Physical Sciences},
    VOLUME = {7},
      YEAR = {2021},
    NUMBER = {1},
     PAGES = {Paper No. 4, 63},
      ISSN = {2524-5317,2199-2576},
   MRCLASS = {35Q31 (35Q35 37G40)},
  MRNUMBER = {4235799},
       DOI = {10.1007/s40818-021-00095-1},
       URL = {https://doi.org/10.1007/s40818-021-00095-1},
}

@article {rodrigo2005,
    AUTHOR = {Rodrigo, Jos\'e{} L.},
     TITLE = {On the evolution of sharp fronts for the quasi-geostrophic
              equation},
   JOURNAL = {Comm. Pure Appl. Math.},
  FJOURNAL = {Communications on Pure and Applied Mathematics},
    VOLUME = {58},
      YEAR = {2005},
    NUMBER = {6},
     PAGES = {821--866},
      ISSN = {0010-3640,1097-0312},
   MRCLASS = {35Q35 (35B10 76B03 76B47 76U05)},
  MRNUMBER = {2142632},
MRREVIEWER = {Drago\c s\ Iftimie},
       DOI = {10.1002/cpa.20059},
       URL = {https://doi.org/10.1002/cpa.20059},
}

@article {clop2022nonlinear,
    AUTHOR = {Clop, Albert and Sengupta, Banhirup},
     TITLE = {Nonlinear transport equations and quasiconformal maps},
   JOURNAL = {Ann. Fenn. Math.},
  FJOURNAL = {Annales Fennici Mathematici},
    VOLUME = {48},
      YEAR = {2023},
    NUMBER = {1},
     PAGES = {375--387},
      ISSN = {2737-0690,2737-114X},
   MRCLASS = {35Q49 (30C62 35F25)},
  MRNUMBER = {4603578},
       DOI = {10.54330/afm.130026},
       URL = {https://doi.org/10.54330/afm.130026},
}

@article{yudovich,
author = {Yudovich, Viktor I.},
year = {1963},
month = {01},
pages = {},
title = {Some bounds for solutions of elliptic equations},
volume = {56},
journal = {Translations. Series 2. American Mathematical Society.}
}

@article{kiselevLuo2022,
  title={Illposedness of ${C}^2$ Vortex Patches},
  author={Kiselev, Alexander and Luo, Xiaoyutao},
  journal={Archive for Rational Mechanics and Analysis},
  volume={247},
  number={3},
  pages={57},
  year={2023},
  publisher={Springer}
}

@article {hmidi2023unifiedtheoryvstatesstructures,
    AUTHOR = {Hmidi, Taoufik and Xue, Liutang and Xue, Zhilong},
     TITLE = {Unified theory on {V}-states structures for active scalar
              equations},
   JOURNAL = {Adv. Math.},
  FJOURNAL = {Advances in Mathematics},
    VOLUME = {486},
      YEAR = {2026},
     PAGES = {Paper No. 110750, 77},
      ISSN = {0001-8708,1090-2082},
   MRCLASS = {35Q35 (35B32 35P30 35Q86 76U05)},
  MRNUMBER = {5007819},
       DOI = {10.1016/j.aim.2025.110750},
       URL = {https://doi.org/10.1016/j.aim.2025.110750},
}

@article{magana2026qgsw,
  author  = {Magaña, Marc and Mateu, Joan and Orobitg, Joan},
  title   = {The regularity of the boundary of vortex patches for the quasi-geostrophic shallow-water equations},
  journal = {arXiv preprint arXiv:2602.22767},
  year    = {2026},
}

@article{TXX2025,
      title={A revisit of patch solutions for the 2{D} {L}oglog-{E}uler type equation}, 
      author={Changhui Tan and Liutang Xue and Zhilong Xue},
      year={2025},
      eprint={2510.18759},
      archivePrefix={arXiv},
      primaryClass={math.AP},
      url={https://arxiv.org/abs/2510.18759},
      journal={arXiv preprint arXiv:2510.18759}
}

@article{VerderaRevisited,
    AUTHOR = {Verdera, Joan},
     TITLE = {The global regularity of vortex patches revisited},
   JOURNAL = {Adv. Math.},
  FJOURNAL = {Advances in Mathematics},
    VOLUME = {416},
      YEAR = {2023},
     PAGES = {Paper No. 108917, 15},
      ISSN = {0001-8708,1090-2082},
   MRCLASS = {35Q31 (35Q35 35Q49 42B20)},
  MRNUMBER = {4549427},
       DOI = {10.1016/j.aim.2023.108917},
       URL = {https://doi.org/10.1016/j.aim.2023.108917},
}

@article {CDE,
    AUTHOR = {Constantin, Peter and Drivas, Theodore D. and Elgindi, Tarek
              M.},
     TITLE = {Inviscid limit of vorticity distributions in the {Y}udovich
              class},
   JOURNAL = {Comm. Pure Appl. Math.},
  FJOURNAL = {Communications on Pure and Applied Mathematics},
    VOLUME = {75},
      YEAR = {2022},
    NUMBER = {1},
     PAGES = {60--82},
      ISSN = {0010-3640},
   MRCLASS = {76D09 (76D17)},
  MRNUMBER = {4373167},
       DOI = {10.1002/cpa.21940},
       URL = {https://doi.org/10.1002/cpa.21940},
}

@book{as,
	author = {Ars\'{e}nio, Diogo and Saint-Raymond, Laure},
	pages = {xii+406},
	publisher = {European Mathematical Society (EMS), Z\"{u}rich},
	series = {EMS Monographs in Mathematics},
	title = {From the {V}lasov-{M}axwell-{B}oltzmann system to incompressible viscous electro-magneto-hydrodynamics. {V}ol. 1},
	year = {2019}}

@article{MiaoTanXueXue2024,
      title={Local regularity and finite-time singularity for a class of generalized {SQG} patches on the half plane}, 
      author={Qianyun Miao and Changhui Tan and Liutang Xue and Zhilong Xue},
      year={2024},
      eprint={2410.19273},
      archivePrefix={arXiv},
      primaryClass={math.AP},
      journal={arXiv preprint arXiv:2410.19273}
}

@article{kiselevLuoAlphaSQG,
   AUTHOR = {Kiselev, Alexander and Luo, Xiaoyutao},
     TITLE = {The {$\alpha$}-{SQG} patch problem is illposed in
              {$C^{2,\beta}$} and {$W^{2,p}$}},
   JOURNAL = {Comm. Pure Appl. Math.},
  FJOURNAL = {Communications on Pure and Applied Mathematics},
    VOLUME = {78},
      YEAR = {2025},
    NUMBER = {4},
     PAGES = {742--820},
      ISSN = {0010-3640,1097-0312},
   MRCLASS = {35R25 (35Q35 35Q86)},
  MRNUMBER = {4863200},
MRREVIEWER = {Jiefeng\ Zhao},
       DOI = {10.1002/cpa.22236},
       URL = {https://doi.org/10.1002/cpa.22236},
}

@article {GancedoUniquenessSQG,
    AUTHOR = {C\'ordoba, Antonio and C\'ordoba, Diego and Gancedo,
              Francisco},
     TITLE = {Uniqueness for {SQG} patch solutions},
   JOURNAL = {Trans. Amer. Math. Soc. Ser. B},
  FJOURNAL = {Transactions of the American Mathematical Society. Series B},
    VOLUME = {5},
      YEAR = {2018},
     PAGES = {1--31},
      ISSN = {2330-0000},
   MRCLASS = {86A10 (35Q35 35Q86)},
  MRNUMBER = {3748149},
MRREVIEWER = {Mikhail\ M.\ Shvartsman},
       DOI = {10.1090/btran/20},
       URL = {https://doi.org/10.1090/btran/20},
}

@article {GNP22,
    AUTHOR = {Gancedo, Francisco and Nguyen, Huy Q. and Patel, Neel},
     TITLE = {Well-posedness for {SQG} sharp fronts with unbounded
              curvature},
   JOURNAL = {Math. Models Methods Appl. Sci.},
  FJOURNAL = {Mathematical Models and Methods in Applied Sciences},
    VOLUME = {32},
      YEAR = {2022},
    NUMBER = {13},
     PAGES = {2551--2599},
      ISSN = {0218-2025,1793-6314},
   MRCLASS = {35Q35 (76U60)},
  MRNUMBER = {4535545},
MRREVIEWER = {Weiliang\ Xiao},
       DOI = {10.1142/S0218202522500610},
       URL = {https://doi.org/10.1142/S0218202522500610},
}

@article {GS14,
    AUTHOR = {Gancedo, Francisco and Strain, Robert M.},
     TITLE = {Absence of splash singularities for surface quasi-geostrophic
              sharp fronts and the {M}uskat problem},
   JOURNAL = {Proc. Natl. Acad. Sci. USA},
  FJOURNAL = {Proceedings of the National Academy of Sciences of the United
              States of America},
    VOLUME = {111},
      YEAR = {2014},
    NUMBER = {2},
     PAGES = {635--639},
      ISSN = {0027-8424,1091-6490},
   MRCLASS = {76S05 (35B35 76D27 76Txx 86A05)},
  MRNUMBER = {3181769},
MRREVIEWER = {Jos\'e\ Miguel\ Pacheco Castelao},
       DOI = {10.1073/pnas.1320554111},
       URL = {https://doi.org/10.1073/pnas.1320554111},
}

@article {KLSplash23,
    AUTHOR = {Kiselev, Alexander and Luo, Xiaoyutao},
     TITLE = {On nonexistence of splash singularities for the
              {$\alpha$}-{SQG} patches},
   JOURNAL = {J. Nonlinear Sci.},
  FJOURNAL = {Journal of Nonlinear Science},
    VOLUME = {33},
      YEAR = {2023},
    NUMBER = {2},
     PAGES = {Paper No. 37, 16},
      ISSN = {0938-8974,1432-1467},
   MRCLASS = {35Q86},
  MRNUMBER = {4554076},
       DOI = {10.1007/s00332-023-09893-2},
       URL = {https://doi.org/10.1007/s00332-023-09893-2},
}

@article {MR4993809,
    AUTHOR = {Zlato{\v s}, Andrej},
     TITLE = {Local regularity and finite time singularity for the
              generalized {SQG} equation on the half-plane},
   JOURNAL = {Duke Math. J.},
  FJOURNAL = {Duke Mathematical Journal},
    VOLUME = {174},
      YEAR = {2025},
    NUMBER = {17},
     PAGES = {3493--3533},
      ISSN = {0012-7094,1547-7398},
   MRCLASS = {35Q35 (76B03)},
  MRNUMBER = {4993809},
       DOI = {10.1215/00127094-2025-0024},
       URL = {https://doi.org/10.1215/00127094-2025-0024},
}

@article {MR4736524,
    AUTHOR = {Jeon, Junekey and Zlato{\v s}, Andrej},
     TITLE = {An improved regularity criterion and absence of splash-like
              singularities for g-{SQG} patches},
   JOURNAL = {Anal. PDE},
  FJOURNAL = {Analysis \& PDE},
    VOLUME = {17},
      YEAR = {2024},
    NUMBER = {3},
     PAGES = {1005--1018},
      ISSN = {2157-5045,1948-206X},
   MRCLASS = {35Q35 (35B65 35Q31 35Q86)},
  MRNUMBER = {4736524},
MRREVIEWER = {Paolo\ Secchi},
       DOI = {10.2140/apde.2024.17.1005},
       URL = {https://doi.org/10.2140/apde.2024.17.1005},
}

@article {MR3666567,
    AUTHOR = {Kiselev, Alexander and Yao, Yao and Zlato{\v s}, Andrej},
     TITLE = {Local regularity for the modified {SQG} patch equation},
   JOURNAL = {Comm. Pure Appl. Math.},
  FJOURNAL = {Communications on Pure and Applied Mathematics},
    VOLUME = {70},
      YEAR = {2017},
    NUMBER = {7},
     PAGES = {1253--1315},
      ISSN = {0010-3640,1097-0312},
   MRCLASS = {35Q35 (35B65 76B03)},
  MRNUMBER = {3666567},
MRREVIEWER = {V\'aclav\ M\'acha},
       DOI = {10.1002/cpa.21677},
       URL = {https://doi.org/10.1002/cpa.21677},
}

@article {MR3549626,
    AUTHOR = {Kiselev, Alexander and Ryzhik, Lenya and Yao, Yao and Zlato{\v s}, Andrej},
     TITLE = {Finite time singularity for the modified {SQG} patch equation},
   JOURNAL = {Ann. of Math. (2)},
  FJOURNAL = {Annals of Mathematics. Second Series},
    VOLUME = {184},
      YEAR = {2016},
    NUMBER = {3},
     PAGES = {909--948},
      ISSN = {0003-486X,1939-8980},
   MRCLASS = {35Q86 (35B65 35Q31 76B03)},
  MRNUMBER = {3549626},
MRREVIEWER = {Gabriela\ Planas},
       DOI = {10.4007/annals.2016.184.3.7},
       URL = {https://doi.org/10.4007/annals.2016.184.3.7},
}

@article {MR3158812,
    AUTHOR = {Elgindi, Tarek Mohamed},
     TITLE = {Osgood's lemma and some results for the slightly supercritical
              2{D} {E}uler equations for incompressible flow},
   JOURNAL = {Arch. Ration. Mech. Anal.},
  FJOURNAL = {Archive for Rational Mechanics and Analysis},
    VOLUME = {211},
      YEAR = {2014},
    NUMBER = {3},
     PAGES = {965--990},
      ISSN = {0003-9527,1432-0673},
   MRCLASS = {35Q31 (35B35 76B03)},
  MRNUMBER = {3158812},
MRREVIEWER = {Raphael\ Stuhlmeier},
       DOI = {10.1007/s00205-013-0691-z},
       URL = {https://doi.org/10.1007/s00205-013-0691-z},
}

@book{schilling2010bernstein,
  author    = {Schilling, Ren{\'e} L. and Song, Renming and Vondra{\v{c}}ek, Zoran},
  title     = {Bernstein Functions: Theory and Applications},
  series    = {De Gruyter Studies in Mathematics},
  volume    = {37},
  publisher = {De Gruyter},
  address   = {Berlin},
  year      = {2010}
}

@article {danchin1997singular,
    AUTHOR = {Danchin, Rapha\"el},
     TITLE = {\'Evolution temporelle d'une poche de tourbillon singuli\`ere},
   JOURNAL = {Comm. Partial Differential Equations},
  FJOURNAL = {Communications in Partial Differential Equations},
    VOLUME = {22},
      YEAR = {1997},
    NUMBER = {5-6},
     PAGES = {685--721},
      ISSN = {0360-5302,1532-4133},
   MRCLASS = {35Q30 (35K55 76C05)},
  MRNUMBER = {1452164},
MRREVIEWER = {Hamid\ Bellout},
       DOI = {10.1080/03605309708821280},
       URL = {https://doi.org/10.1080/03605309708821280},
}
    
\end{document}